\documentclass{article}
\usepackage{latexsym, amscd, amsfonts, eucal, mathrsfs, amsmath, amssymb, amsthm, xypic,xr, makeidx, stmaryrd, color, enumerate, tikz}
\usepackage{appendix, multicol}
\usepackage[all]{xy}
\usepackage{hyperref}
\newcommand*\circleit[1]{\tikz[baseline=(char.base)]{
            \node[shape=circle,draw,inner sep=2pt] (char) {#1};}}

\newtheorem{theorem}{Theorem}[section]
\newtheorem{lemma}[theorem]{Lemma}
\newtheorem{proposition}[theorem]{Proposition}
\newtheorem{corollary}[theorem]{Corollary}

\theoremstyle{definition}
\newtheorem{definition}[theorem]{Definition}

\newtheorem{example}[theorem]{Example}

\newtheorem{remark}[theorem]{Remark}

\usepackage{cleveref}
\begin{document}

\title{Orientations and Topological Modular Forms with Level Structure}
\author{Dylan Wilson\footnote{The author was supported by NSF grant DGE-1324585 while completing this work.}}
\maketitle
\begin{abstract} Using the methods of Ando-Hopkins-Rezk, we describe the characteristic series arising from $E_\infty$-genera valued in topological modular forms with level structure. We give examples of such series for $tmf_0(N)$ and show that the Ochanine genus comes from an $E_\infty$-ring map. We also show that, away from 6, certain $tmf$ orientations of $MString$ descend to orientations of $MSpin$. 
\end{abstract}
\tableofcontents
\newpage
\section*{Introduction}
One of our jobs as topologists is to study invariants of manifolds. The most accessible of these are cobordism invariants, called {\it genera}, which are ring maps
\[
\phi: MSO_* \longrightarrow R
\]
Here $MSO$ is the Thom spectrum with coefficients the ring of cobordism classes of oriented-manifolds. Precomposing with the forgetful map $MU_* \longrightarrow MSO_*$ this genus determines a formal group law over $R$. If $R$ is a $\mathbb{Q}$-algebra, the genus is then determined by the logarithm of this formal group law which takes the form
\[
\log_\phi(x) = \sum_{i\ge 0} \frac{\phi(\mathbb{C}^{2i})}{2i+1} x^{2i+1}
\]
or equivalently by the Hirzebruch characteristic series $K_\phi(u) = \dfrac{u}{\log_\phi^{-1}(u)}$. 
\begin{example} The $L$-genus (i.e. the signature) with characteristic series
\[
K_{\textup{Sign}}(u) = \exp\left(\sum_{k\ge 2} \frac{2^{k+1}(2^{k-1}-1)B_k}{k} \frac{u^k}{k!}\right)
\]
where the Bernoulli numbes $B_k$ are defined by the generating series $\sum_{k\ge 0} B_k x^k/k! = \dfrac{x}{e^x -1}$. 
\end{example}
\begin{example} The $\widehat{A}$-genus, with characteristic series
\[
K_{\widehat{A}}(u) = \exp\left(-\sum_{k\ge 2} \frac{B_k}{k} \frac{u^k}{k!}\right)
\]
\end{example}
The first two genera have the useful property that they vanish on projective bundles associated to even-dimensional complex vector bundles. Such genera were classified in a beautiful theorem:
\begin{theorem}[Ochanine] The logarithm of a genus $\phi$ that vanishes on projective bundles of even-dimensional complex vector bundles is always of the form
\[
\textup{log}_{\phi}(x) = \int_0^x \frac{du}{\sqrt{1 - 2\delta u^2 + \epsilon u^4}}
\]
where $\delta, \epsilon \in R$. 
\end{theorem}

In particular, there is a universal such with target the ring $\mathbb{Q}[\delta, \epsilon]$ of level 2 modular forms over $\mathbb{Q}$ called the \emph{Ochanine genus}. The characteristic series was computed by Zagier \cite{Zag}:
\[
K_{\textup{Och}}(u) = \textup{exp}\left(\sum_{k \ge 2} 2\widetilde{G}_{k} \frac{u^k}{k!}\right)
\]
where the $\widetilde{G}_{k}$ are certain Eisenstein series of level 2. 

Atiyah-Bott-Shapiro showed that the $\widehat{A}$-genus on Spin-manifolds arose from a finer, integral invariant on families, neatly expressed by a map of spectra
\[
\textup{MSpin} \longrightarrow KO
\] 
and this motivated the search for a cohomology theory to be the target of the Ochanine genus. Such a theory was built by Landweber-Ravenel-Stong and dubbed elliptic cohomology. Around the same time, Witten introduced a genus on Spin-manifolds that took values in the ring of quasi-modular forms. He proved that when the Spin-manifold satisfies $p_1(TM) = 0$ the genus is actually valued in modular forms. The characteristic series for this genus was also computed by Zagier \cite{Zag} in terms of the classical Eisenstein series $G_k$:
\[
K_{\textup{Wit}}(u) = \textup{exp}\left(\sum_{k \ge 2} 2G_{k}\frac{u^k}{k!}\right)
\]
The requirement on the first Pontryagin class comes from the fact that $G_2$ is not a modular form.

In Hopkins' 1994 ICM address he proposed that the Witten genus comes from a map of spectra into the then conjectural spectrum of topological modular forms
\[
\sigma: \textup{MString} \longrightarrow tmf
\]
where $MString_*$ is the cobordism ring for Spin manifolds such that the generator $p_1/2\in H^4(Bspin, \mathbb{Z})$ vanishes on their stable normal bundle. 

Since then, work of Ando, Goerss, Hopkins, Mahowald, Rezk, Strickland, and others has culminated in a construction of this map of $E_\infty$-ring spectra. Furthermore, Ando-Hopkins-Rezk \cite{AHR} determined all possible characteristic series of $E_\infty$-genera on String manifolds valued in topological modular forms. Their method of proof sheds light on the appearance of the Bernoulli numbers and the Eisenstein series in the formulae for characteristic series above: it is crucial that both sequences can be $p$-adically interpolated. 

Their work settles the question of genera taking values in modular forms of level 1, but leaves open questions about genera valued in higher level modular forms, including the Ochanine genus and others constructed in work of Hirzebruch \cite{Hir}.

Thanks to work of Hill and Lawson \cite{HL}, we now have candidates for $E_\infty$-rings that deserve to be called topological modular forms with level structure, and we can ask about constructing $E_\infty$-maps out of Thom spectra
\[
MG \longrightarrow tmf(\Gamma)
\]
The purpose of this paper is to apply the machinery of \cite{AHR} and \cite{ABGHR} to the case of topological modular forms with level structure and construct various genera. In order to state the main result we will recall a bit of notation. If $\Gamma \subset GL_2(\mathbb{Z}/N)$ there is a notion of modular forms with level structure $\Gamma$. When
\[
\Gamma = \Gamma_1(N):= \left\{ \begin{pmatrix} 1&b\\ 0&1\end{pmatrix} : \right\}
\]
then the group $(\mathbb{Z}/N)^{\times}$ acts on $MF_*(\Gamma_1(N))$ and there is a nice formula for the operation $\psi^p$ in terms of $q$-expansions. Indeed, if $g_k \in MF_k(\Gamma_1(N))$ has a $q$-expansion given by $g_k(q) = \sum a_nq^n$ and the action of $p \in (\mathbb{Z}/N)^{\times}$ is given by
\[
(g_k)\vert_{\langle p\rangle}(q) = \sum_{n\ge 0} b_n q^n
\]
then we have:
\[
(\psi^pg_k)(q) = p^k\sum_{n \ge 0} b_nq^{pn}
\]
All of this will be reviewed in detail below, but for now we at least state the main result. 
\begin{theorem}\label{theorem:main} Let $\Gamma \subset GL_2(\mathbb{Z}/N)$ be a level structure, and consider a sequence of modular forms $\{g_k\}_{k \ge 2} \in MF_*(\Gamma) \otimes \mathbb{Q}$. Then 
\begin{enumerate}
\item There exists an $E_\infty$-ring map $\textup{MString}[1/N] \longrightarrow tmf(\Gamma)$ with associated characteristic series $\exp(2\sum_{k \ge 4}g_k \frac{u^k}{k!})$ if and only if 
\begin{enumerate}
\item For each $i\ge 2$, $g_{2i+1} = 0$, 
\item For every prime $p \!\not\vert N$ and every unit $\lambda \in \mathbb{Z}_p^{\times}/\{\pm 1\}$ the sequence of rational $p$-adic modular forms  $\{(1-\lambda^k)(1-\frac{1}{p}\psi^p)g_k\}_{k \ge 4} \in MF_{p, *}(\Gamma) \otimes \mathbb{Q}$ satisfies the generalized Kummer congruences (\ref{definition:kummer}),
\item For every prime $p \not\vert N$, $T_pg_k = (1+p^{k-1})g_k$, where $T_p$ is the $p$th Hecke operator,
\item We have the congruence $g_k \equiv G_k$ \textup{mod} $MF_*(\Gamma, \mathbb{Z}[1/N, \zeta_N])$, where $G_k$ is the unnormalized Eisenstein series of weight $k$. 
\end{enumerate}
\item Suppose that $2\vert N$. There exists an $E_\infty$-ring map $\textup{MSpin}[1/N] \longrightarrow tmf(\Gamma)$ with associated characteristic series $\exp(2\sum_{k \ge 2}g_k \frac{u^k}{k!})$ if and only if the conditions (a)-(c) above are satisfied for $k\ge 2$ and we have the congruence
\[
g_k \equiv \widetilde{G}_k \textup{ mod } MF_*(\Gamma, \mathbb{Z}[1/N, \zeta_N])
\]
where $\widetilde{G}_k$ is the level 2 modular form with $q$-expansion $\widetilde{G}_k = -\dfrac{B_k}{2k} + \sum_{n\ge 1} q^n \sum_{d\vert n} (-1)^{n/d}d^{k-1}$. 
\end{enumerate}
In each instance the set of homotopy classes of $E_\infty$-ring maps corresponding to a given characteristic series is a non-empty torsor for the group
\[
[\Sigma KO^{\wedge}_2, L_{K(1)}L_{K(2)}tmf(\Gamma)^{\wedge}_2]
\]
which has exponent at most 2. In particular, upon inverting $2$, the homotopy class of an $E_\infty$-genus valued in $tmf(\Gamma)$ is determined by its characteristic series. 
\end{theorem}
\begin{corollary} Up to homotopy, there is a unique $E_\infty$-ring map $\textup{MSpin}[1/2] \longrightarrow tmf_0(2)$ refining the Ochanine genus. 
\end{corollary}
\begin{remark} This answers in the affirmative the Question 1.1 of \cite{HL}.
\end{remark}
A more careful analysis at the prime $2$ reveals the following
\begin{theorem} Let $\textup{MSwing}$ denote the Thom spectrum associated to the fiber of the map
\[
bspin \stackrel{w_4}{\longrightarrow} \Sigma^4H\mathbb{Z}/2
\]
Then Theorem \ref{theorem:main} holds mutatis mutandis for $E_\infty$-ring maps $\textup{MSwing}[1/N] \longrightarrow tmf(\Gamma)$ when $3 \vert N$. 
\end{theorem} 
\begin{remark} It seems plausible that such orientations reproduce the real $E$-theory orientations found in \cite{KS}. 
\end{remark}
A curious corollary of the proof of the main theorem is the following
\begin{theorem} There is an $E_\infty$-ring map
\[
\textup{MSpin} \longrightarrow tmf[1/6]
\]
\end{theorem}
\begin{remark} No such map exists refining the Witten genus, even as homotopy commutative spectra. See Theorem \ref{theorem:tmf-away-from-6} below. 
\end{remark}
We now turn to a more precise overview of our results and outline of the proofs. In \S 1 we give a quick introduction to topological modular forms and elliptic curves with level structure, and review the obstruction theory for $E_\infty$-genera given in \cite{ABGHR}. Our problem is to describe the space of dotted arrows
\[
\xymatrix{
g \ar[r] & gl_1S \ar[r]\ar[d] & gl_1S/g \ar@{-->}[dl]\\
& gl_1tmf(\Gamma)
}
\]
where $g$ is a delooping of either $Spin$ or $String$, $gl_1R$ denotes the spectrum of units for an $E_\infty$ ring, and the top line is a fiber sequence. 

In \S 2 we analyze the substantially easier problem of parameterizing the space of dotted arrows
\[
\xymatrix{
g \ar[r] & gl_1S \ar[r]\ar[d] & gl_1S/g \ar@{-->}[dl]\\
& L_{K(1)}gl_1tmf(\Gamma)^{\wedge}_p
}
\]
where $p$ is a prime not dividing the level, and we have localized with respect to Morava $K$-theory. Following Ando-Hopkins-Rezk, we use the solution to the Adams conjecture to reduce this to a computation of $L_{K(1)}tmf(\Gamma)^0(bspin)$ which we give in terms of measures on $\mathbb{Z}_p^{\times}/\{\pm 1\}$ valued in $p$-adic modular forms of level $\Gamma$. This sheds light on the appearance of Eisenstein series and Bernoulli numbers in the characteristic series of these genera. We have included an especially detailed description of this computation as it has not yet appeared elsewhere in the literature, even for the case of $tmf$. 

In \S 3 we review Rezk's logarithm, construct a topological lift of the Atkin operator, and describe the space of dotted arrows
\[
\xymatrix{
g \ar[r] & gl_1S \ar[r]\ar[d] & gl_1S/g \ar@{-->}[dl]\\
& L_{K(1)\vee K(2)}gl_1tmf(\Gamma)^{\wedge}_p
}
\]
as certain measures valued in $p$-adic modular forms of level $\Gamma$ which are fixed by the Atkin operator. Up to now there is no significant difference from the case of $tmf$. When replacing the $K(1)\vee K(2)$-local case with the $p$-complete case, however, our story diverges from that of \cite{AHR}. Indeed, the difference is measured by the fiber $F$ in the sequence
\[
F \longrightarrow gl_1tmf(\Gamma)^{\wedge}_p\longrightarrow L_{K(1)\vee K(2)}gl_1tmf(\Gamma)^{\wedge}_p
\]
and the main obstruction to an orientation is the map
\[
[g, F] \longrightarrow [g, gl_1tmf(\Gamma)^{\wedge}_p]
\]
When $g = string$ we are in the case analyzed by \cite{AHR} and one can prove that the group $[g, F]$ is zero. However, when $g = spin$, this group {\it never} vanishes, so one must analyze the map itself. In \S4 we investigate the spectrum $F$ and the above map. Finally, in \S5 we prove the main theorem and build the examples described above. We expect the same methods to produce $E_\infty$-ring maps
\[
\textup{MU} \longrightarrow tmf_1(N)
\]
refining Hirzebruch's level $N$ genera. 
\newline

\textbf{Acknowledgements}. I would like to thank Tyler Lawson for asking me whether the Ochanine genus comes from an $E_\infty$-ring map. I owe a huge intellectual debt to Matt Ando, Mike Hopkins, and Charles Rezk. They were each very helpful and graciously answered my questions about their paper. I would also like to thank Joel Specter and Frank Calegari for patiently explaining the theory of $p$-adic modular forms. Finally, this paper would not exist without the help of my advisor, Paul Goerss, to whom I am very grateful. 

\section{Background on Orientations and Level Structures}
\subsection{Orientations} The spectrum analog of the group algebra functor, $\Sigma_+^{\infty}\Omega^{\infty}$, participates in an adjunction (\cite[5.2]{ABGHR2})
\[
\Sigma_+^\infty\Omega^\infty: \textup{Spectra}_{\ge 0} \leftrightarrows E_{\infty}\textup{-Ring} : gl_1
\]
from connective spectra to $E_\infty$-rings. If $X$ is a space and $R$ an $E_{\infty}$-ring, then
\[
gl_1(R)^0(X) = R^0(X)^{\times}
\]
In the special case when $X = S^n$, $n>0$, and we use {\it reduced} homology we get an identification of $\pi_n(gl_1R)$ with the group $$\{1+x \in \left(\pi_nR\right)^{\times}: x \in \pi_nR\}.$$ We will make heavy use of the fact that the assignment $1+x \mapsto x$ gives an equivalence
\[
\pi_qgl_1R \cong \pi_qR, \quad q>0
\]
though this is {\it not} induced by a map of spectra in general. 

For a connective spectrum $g$, we will occasionally denote the suspension $\Sigma g$ by $bg$, and their zeroth space by $G$ and $BG$ respectively. (The only exceptions to this convention are the connective $K$-theory spectra $bo$ and $bu$ whose zeroth spaces are $BO \times \mathbb{Z}$ and $BU \times \mathbb{Z}$, respectively.) 

Given a map
\[
f: bg \longrightarrow bgl_1S
\] 
we can form the associated Thom spectrum, which can be defined as the homotopy pushout of $E_\infty$-rings
\[
\xymatrix{
\Sigma_+^{\infty}BG \ar[r]\ar[d] & \Sigma_+^\infty BGL_1S\ar[d]\\
S \ar[r] & MG}
\]
One can show that the underlying spectrum agrees with the usual notion of Thom spectrum in the examples of interest (see, e.g., the comparison results in \cite[\S3]{ABGHR1} for the connection with the Thom spectra constructed in Lewis's thesis \cite[Ch. IX]{LMS}).  An immediate corollary of this definition is the following:
\begin{theorem}[{\cite[4.3]{ABGHR}}]\label{theorem:ABGHR obstruction} Let $R$ be an $E_\infty$-ring with unit map $\iota: S \longrightarrow R$. Then there is a homotopy pullback diagram of spaces
\[
\xymatrix{
\textup{Map}_{\textup{CAlg}}(MG, R) \ar[r]\ar[d] & \textup{Map}_{\textup{Spectra}}(gl_1S/g, gl_1R)\ar[d]\\
\{gl_1\iota\} \ar[r] & \textup{Map}_{\textup{Spectra}}(gl_1S, gl_1R)
}
\]
and any choice of point in $\textup{Map}_{\textup{CAlg}}(MG, R)$ gives an equivalence $$\textup{Map}_{\textup{Spectra}}(bg, gl_1R) \stackrel{\cong}{\longrightarrow} \textup{Map}_{\textup{CAlg}}(MG, R)$$
\end{theorem} 
In practice, we interpret this theorem as the statement that the space of $E_\infty$ $R$-orientations of $MG$ is the space of dotted arrows in the diagram
\[
\xymatrix{
g \ar[r] & gl_1S \ar[r]\ar[d] & gl_1S/g \ar@{-->}[dl]\\
& gl_1R &
}
\]
In particular, the obstruction to the existence of an orientation is a class in $[g, gl_1R]$. We will have occasion to consider more general targets in this diagram, so we introduce a notation for this.
\begin{definition} Let $\epsilon: gl_1S \longrightarrow X$ be a spectrum under $gl_1S$, then we denote by $\textup{Nulls}(g, X)$ the homotopy pullback
\[
\xymatrix{
\textup{Nulls}(g,X) \ar[r]\ar[d] & \textup{Map}_{\textup{Spectra}}(gl_1S/g, X)\ar[d]\\
\{\epsilon\} \ar[r] & \textup{Map}_{\textup{Spectra}}(gl_1S, X)
}
\]
In particular, $\textup{Map}_{CAlg}(MG, R) = \textup{Nulls}(g,gl_1R)$. 
\end{definition}
By construction this gives a homotopy limit preserving functor
\[
\textup{Nulls}(g, -): \textup{Spectra}_{gl_1S/} \longrightarrow \textup{Spaces}
\]
For the remainder of the paper we will assume that we are in the geometric situation where the map $bg \longrightarrow bgl_1S$ factors through the spectrum $bso$. In this setting we can recover Hirzebruch's characteristic series when $R$ is a rational $E_\infty$-ring. Indeed, in this case, we have a canonical orientation 
\[
MG \longrightarrow MSO \longrightarrow H\mathbb{Q} \stackrel{\cong}{\leftarrow} S_{\mathbb{Q}} \longrightarrow R
\]
Thus, by Theorem \ref{theorem:ABGHR obstruction}, any other element in $\pi_0\textup{Map}_{\textup{CAlg}}(MG, R)$ is specified by an element in $[bg, gl_1R]$. However, since $R$ is rational, there is a logarithm
\[
\ell_0: \tau_{\ge 1}(gl_1R) \stackrel{\cong}{\longrightarrow} \tau_{\ge 1}R
\]
where $\tau_{\ge n}$ denotes the truncation functor from spectra to spectra with homotopy concentrated in degrees at least $n$. Since $g$ is connective we deduce that orientations are determined by difference classes in $[bg, R]$. 

For $n \ge 1$ we denote by $bo\langle 2n\rangle:= \tau_{\ge 2n}bo$ the $2n$-connective cover of $bo$. We also observe the following conventions:
\[
bso:= bo\langle 2\rangle, \quad bspin := bo\langle 4\rangle, \quad bstring:=bo\langle 8\rangle \cong \Sigma^8bo
\]
Recall that we have
\[
\pi_*bu = \mathbb{Z}[v], \quad \pi_*bo = \mathbb{Z}[\eta, x_4, v_{\mathbb{R}}]/(2\eta, \eta x_4, x_4^2-4v_{\mathbb{R}}, \eta^3)
\]
Complexification of vector bundles induces a map $c: bo \longrightarrow bu$, and forgetting structure induces a map (which does not give a ring homomorphism on homotopy groups) $r: bu \longrightarrow bo$. On homotopy these are
\[
c_*v_{\mathbb{R}} = v,\,\, c_*x_4 = 2v^2\quad \textup{and}\quad r_*v^4k = 2v_{\mathbb{R}}^{k}, \,\,r_*v^{4k+2} = x_4v_{\mathbb{R}}^{k},\,\, r_*v^{\textup{odd}} = 0
\]

\begin{theorem}[{\cite{AHR}}] Let $R$ be a rational $E_\infty$-ring and $n\le 4$. There is a natural isomorphism
\[
\pi_0\textup{Map}_{\textup{CAlg}}(MO\langle 2n\rangle, R) \stackrel{\cong}{\longrightarrow} \prod_{k\ge n} \pi_{2k}R
\]
where the $k$th component $t_k$ is given by
\[
S^{2k} \stackrel{v^k}{\longrightarrow} bu\langle 2n\rangle \longrightarrow bo\langle 2n\rangle \stackrel{\delta}\longrightarrow \tau_{\ge 1}gl_1R \cong \tau_{\ge 1}R
\] 
(which is zero if $k$ is odd). The Hirzebruch characteristic series is then $K(u) = \exp\left(\sum_{k \ge n} t_k \dfrac{u^k}{k!}\right)$. 
\end{theorem}

We now return to the case where $R$ is an arbitrary $E_\infty$-ring. The localization $R \longrightarrow R_{\mathbb{Q}}$ gives a map
\[
\pi_0\textup{Map}_{\textup{CAlg}}(MO\langle 2n\rangle, R) \longrightarrow \prod_{k \ge n} \pi_{2k}R \otimes \mathbb{Q}
\]
It turns out (\cite[5.14]{AHR}) that the elements $\{(t_k/2)\}: = \{b_k\} \in \pi_{2k}R \otimes \mathbb{Q}$ are independent of the orientation modulo $\mathbb{Z}$, which motivates the following. 
\begin{definition}\label{definition:char} We denote the image of the above map postcomposed with division by 2 as $\textup{Char}(\textup{Spin}, R)$ when $n=2$. This is the set of characteristic series of $E_\infty$ Spin genera valued in $R$ under the correspondence $\{b_k\} \mapsto \exp\left(2\sum b_k u^k/k!\right)$. We define $\textup{Char}(\textup{String}, R)$ analogously for $n=4$. 
\end{definition}
\subsection{Level Structures and Modular Forms} This section exists mainly to establish notation and provide references for precise definitions and results. Intimate knowledge of algebraic stacks is not necessary to understand the main results of the paper, nor their proofs. Our main reference is \cite{Con}, but most of the results and definitions were originally written down in \cite{KM} and \cite{DR}. 

A {\it generalized elliptic curve} \cite[2.1.4]{Con} over a base scheme $S$ is a separated, flat map $\pi: E \longrightarrow S$ such that all geometric fibers are either elliptic curves or N\'eron polygons \cite[II.1.1]{DR} together with a morphism $+: E^{\textup{sm}}\times_S E \longrightarrow E$ and a section $e \in E^\textup{sm}(S)$ in the smooth locus. We require that this data gives $E^{\textup{sm}}$ the structure of a commutative group scheme with identity $e$ and defines an action of $E^{\textup{sm}}$ on $E$ that acts by rotations on the graph of irreducible components on each singular fiber. 

A $\Gamma(N)$-{\it structure} \cite[2.4.2]{Con} on a generalized elliptic curve is an ordered pair $(P,Q)$ with $P,Q \in E^{\textup{sm}}[N](S)$ such that 
\begin{enumerate}
\item the rank $N^2$ Cartier divisor
\[
D = \sum_{i, j \in \mathbb{Z}/N\mathbb{Z}} [iP+jQ]
\]
is a subgroup scheme killed by $N$
\item $D$ meets all irreducible components of all geometric fibers of $E$. 
\end{enumerate}
In the case where $E = E^{\textup{sm}}$ and $N$ is invertible in $\mathcal{O}_S$, this is equivalent to the data of an isomorphism of constant group schemes $(\mathbb{Z}/N\mathbb{Z})^2 \cong E[N]$ (after finite \'etale base change). 

\begin{remark} When doing homotopy theory, we will almost always be interested in the case when $N$ is invertible, since we must invert $N$ to get $E_\infty$-versions of these moduli problems. However, to define Hecke operators (i.e. power operations) as in \S3.1-3.2, we pass through moduli problems defined over bad primes. It is possible to avoid these more arithmetic moduli entirely through ad-hoc constructions, but the author feels the present approach is more conceptual.
\end{remark}

\begin{theorem}[Katz-Mazur, Deligne-Rappaport, Conrad] The moduli problem of generalized elliptic curves with $\Gamma(N)$-structure is representable by a proper, flat Deligne-Mumford stack $\mathscr{M}_{\Gamma(N)}$ which is CM of relative dimension 1 over $\mathbb{Z}$. Moreover, the restriction to $\mathbb{Z}[1/N]$ is smooth.  
\end{theorem}
For each $N\ge 1$, there is a finite, faithfully flat map \cite[4.1.1]{Con}
\[
\mathscr{M}_{\Gamma(N)} \longrightarrow \mathscr{M}_{\Gamma(1)}:=\mathscr{M}_{\textup{ell}}
\]
which forgets the data of $P$ and $Q$ and collapses N\'eron $n$-gon fibers to a nodal cubic fibers. 

There is an action of $GL_2(\mathbb{Z}/N\mathbb{Z})$ on the universal generalized elliptic curve with $\Gamma(N)$-structure $\mathscr{E}_{\Gamma(N)}\longrightarrow \mathscr{M}_{\Gamma(N)}$ given by modifying the points $(P,Q)$. Unless otherwise specified, from now on $\Gamma$ will denote one of the following subgroups:
\[
\Gamma_1(N) : = \left\{ \begin{pmatrix}1 & *\\
0 & *\end{pmatrix}\right\},\quad\quad
\Gamma_0(N) : = \left\{ \begin{pmatrix}* & *\\
0 & *\end{pmatrix}\right\}
\]
When $N$ is square-free, we denote by $\mathscr{M}_{\Gamma}$ the normalization of  $\mathscr{M}_{\Gamma(1)}$ in the quotient stack $\mathscr{M}^{\circ}_{\Gamma(N)}[1/N]/\Gamma$. Here $\mathscr{M}^{\circ}_{\Gamma(N)}$ denotes the open complement of the nonsingular locus of the universal generalized elliptic curve. When $N$ is not square-free, one can still produce algebraic stacks by giving a modular description and proving representability, but we will not need these here. 

These stacks classify generalized elliptic curves together with a point of `exact order $N$' in the sense of Drinfeld (resp. a subgroup scheme of order $N$) which must satisfy conditions on the singular fibers \cite[4.1.5]{Con}. For smooth elliptic curves over a base where $N$ acts invertibly, one can remove the quotation marks. While $\mathscr{M}_{\Gamma_1(N)}$ is a Deligne-Mumford stack, the same is not true in general for $\mathscr{M}_{\Gamma_0(N)}$, though it is so if $N$ is square-free or inverted \cite[3.1.7]{Con}.  

Let $\mathscr{E}_{\Gamma(N)} \longrightarrow \mathscr{M}_{\Gamma(N)}$ denote the universal generalized elliptic curve. Denote by $\omega_{\Gamma(N)}$ (or just $\omega$ if there is no confusion) the relative dualizing sheaf for the map $\mathscr{E}_{\Gamma(N)} \longrightarrow \mathscr{M}_{\Gamma(N)}$. This turns out to be a line bundle, and can be equivalently defined as the dual of the Lie algebra of $\mathscr{E}_{\Gamma(N)}^{\textup{sm}}$, i.e. $e^*(\Omega^1_{\mathscr{E}_{\Gamma(N)}/\mathscr{M}_{\Gamma(N)}})$. 

\begin{definition} The group of {\it modular forms of level $\Gamma$ and weight $k$} over a ring $R$ is defined as
\[
M_k(\Gamma, R):= H^0(\mathscr{M}_{\Gamma} \otimes R, \omega^{\otimes k})
\]
\end{definition}
By flat base change, if $S$ is a flat $R$-algebra, we have
\[
M_k(\Gamma, S) = M_k(\Gamma, R)\otimes_RS
\]
In particular, $M_k(\Gamma, \mathbb{Z}) \otimes \mathbb{C} \cong M_k(\Gamma, \mathbb{C})$ and one can check that this latter group recovers the classical group of holomorphic modular forms. 

The Fourier expansion of a holomorphic modular form can be recovered algebraically using the theory of {\it Tate curves}. First, one may define an elliptic curve of the form $\mathbb{G}_m/q^{\mathbb{Z}}$ over $\mathbb{Z}((q))$. This has a unique extension to a generalized elliptic curve over $\mathbb{Z}[[q]]$ with special fiber a nodal curve. This is called the Tate curve, $\textup{Tate}(q)$, and comes with a canonical isomorphism of formal groups
\[
\widehat{\mathbb{G}}_m \cong \textup{Tate}(q)^{\wedge}_{q=0}
\]
For each $N \ge 1$, there is a unique generalized elliptic curve $\textup{Tate}_N$ over $\mathbb{Z}[[q^{1/N}]]$ restricting to $\textup{Tate}_1\otimes_{\mathbb{Z}[[q]]}\mathbb{Z}((q^{1/N}))$ over $\mathbb{Z}((q^{1/N}))$ and having $n$-gon fiber at $q^{1/N}= 0$. This curve comes with a canonical $\Gamma_1(N)$ structure (corresponding to $q^{1/N}$) classified by a map
\[
\textup{Spec}(\mathbb{Z}[[q^{1/N}]]) \longrightarrow \mathscr{M}_{\Gamma_1(N)}
\]
Evaluation of a modular form at this curve is called the {\it evaluation at the cusp}, and, using the canonical trivialization of $\omega$ on the Tate curve, corresponds to a power series in $q$. This agrees with the Fourier expansion over $\mathbb{C}$.

The map to the $fpqc$ algebraic stack of formal groups
\[
\mathscr{M}_{\Gamma} \longrightarrow \mathscr{M}_{FG}
\]
classifying the formal group associated to the universal generalized elliptic curve is flat and so defines a diagram of weakly even periodic Landweber exact cohomology theories parameterized by flat maps 
\[
\textup{Spec}(R) \longrightarrow \mathscr{M}_{\Gamma}
\]
\begin{theorem}[Goerss-Hopkins-Miller, Hill-Lawson] There is a sheaf of $E_\infty$-rings $\mathcal{O}^{\textup{top}}$ on the \'etale site of $\mathscr{M}_{\Gamma}[1/N]$ rigidifying the given diagram of Landweber exact cohomology theories. 
\end{theorem}
We adopt the following notations:
\[
TMF(\Gamma):= \mathcal{O}^{\textup{top}}(\mathscr{M}^{\circ}_{\Gamma}[1/N]), \quad Tmf(\Gamma) := \mathcal{O}^{\textup{top}}(\mathscr{M}_{\Gamma}[1/N]), \quad tmf(\Gamma):= \tau_{\ge 0}Tmf(\Gamma)
\]
\begin{remark} The definition of $tmf(\Gamma)$ as a connective cover is known to be unsatisfactory when the genus of the modular curve $\mathscr{M}_{\Gamma}$ is bigger than one. This doesn't affect the construction of our genera, but instead suggests that, in these cases, the genera should have more severe restrictions on their image.  
\end{remark}

\section{The $K(1)$-local Case}
Fix a level $\Gamma \in \{\Gamma_0(N), \Gamma_1(N)\}$ and a prime $p$ not dividing $N$. The goal of this section is to prove the following theorem. All the undefined terms and notation will be explained below.
\begin{theorem}\label{theorem:K(1)-local parameterization} The image of the map
\[
\pi_0\textup{Nulls}(spin, L_{K(1)}gl_1tmf(\Gamma)^{\wedge}_p) \longrightarrow \prod_{k \ge 2}MF_{p,k}(\Gamma)\otimes_{\mathbb{Z}_p} \mathbb{Q}_p
\]
is an isomorphism onto the set of sequences $\{g_k\}_{k\ge 2}$ of $p$-adic modular forms satisfying the following conditions:
\begin{enumerate}
\item For all $k\ge 1$, $g_{2k+1} = 0$,
\item For all $c \in \mathbb{Z}_p^{\times}$, the sequence $\{(1-c^k)g^{(p)}_k\}$ extends to a measure $\mu_c$ on $\mathbb{Z}^{\times}_p/\{\pm 1\}$ valued in $V_{\infty}(\Gamma)$,
\item For all $c \in \mathbb{Z}_p^{\times}$, we have $\int_{\mathbb{Z}^{\times}_{p}/\{\pm 1\}} 1 \, d\mu_c = \frac{1}{2p}\log(c^{p-1})$.
\end{enumerate}
\end{theorem}
\begin{remark} As a corollary of \ref{theorem:units}, below, we have $\textup{Nulls}(spin, L_{K(1)}gl_1tmf(\Gamma)^{\wedge}_p) \cong \textup{Nulls}(spin, gl_1L_{K(1)}tmf(\Gamma)^{\wedge}_p)$ so this is actually a description of $K(1)$-local orientations.
\end{remark}
In order to prove this theorem we need a good understanding of $[gl_1S/spin, L_{K(1)}gl_1tmf]$. It turns out that, as a consequence of the Adams conjecture and the computation of the $K(1)$-local sphere by Adams and Mahowald, we have
\[
L_{K(1)}(gl_1/spin) \cong KO_p
\]
Together with Rezk's logarithm, this reduces much of the work to understanding $[KO_p, L_{K(1)}tmf(\Gamma)]$. We will compute this in \S2.2
 using the $K(1)$-local Adams-Novikov spectral sequences, after collecting the necessary preliminaries on $p$-adic modular forms in \S2.1. Finally, we will translate our computation into the language of measures and complete the proof following \cite{AHR}.
\subsection{Complements on $p$-adic modular forms}
For this section we fix an integer $N\ge 1$, $\Gamma \in \{\Gamma(N), \Gamma_1(N), \Gamma_0(N)\}$, and a prime $p$ not dividing $N$. The Hasse invariant, $A$, is a level 1 modular form that takes the value $1$ on the mod $p$ Tate curve. Thus, from the point of view of $q$-expansions, $A$ behaves like a unit mod $p$. It is essentially for this reason that the theory of $p$-adic modular forms requires that we invert a lift of (some power) of the Hasse invariant. This corresponds to restricting attention to those elliptic curves with height 1 formal group. 

Recall that we have a map of $fpqc$-stacks
\[
\mathscr{M}_{\Gamma} \longrightarrow \mathscr{M}_{FG}
\]
classifying the formal group of the universal generalized elliptic curve. Let $\mathscr{M}_{FG}^{\ge n}\subset \mathscr{M}_{FG}\otimes \mathbb{Z}_{(p)}$ denote the regularly embedded, closed substack classifying formal groups of height at least $n$, and denote by $\mathscr{M}^{<n}_{FG}$ the open complement. 

We define the {\it ordinary locus} of $\left(\mathscr{M}_{\Gamma}\right)^{\wedge}_p$ by the pullback of formal stacks
\[
\xymatrix{
\mathscr{M}^{\textup{ord}}_{\Gamma}\ar[r] \ar[d]& \left(\mathscr{M}_{\Gamma}\right)^{\wedge}_p\ar[d]\\
\mathscr{M}^{<2}_{FG}\ar[r] & \left(\mathscr{M}_{FG}\right)^{\wedge}_p
}
\]
This classifies elliptic curves over complete, local Noetherian algebras with residue field of characteristic $p$ whose formal group is a deformation of a height 1 formal group. It is the open complement of the supersingular locus (the pullback of $\mathscr{M}^{\ge 2}_{FG}$). 

\begin{definition} Let $R$ be a $p$-complete $\mathbb{Z}_p$-algebra. We define the ring of {\it$p$-adic modular forms of level $\Gamma$ and weight $k$} over $R$ by
\[
MF_{p,k}(\Gamma, R) := H^0(\mathscr{M}^{\textup{ord}}_{\Gamma}\widehat{\otimes} R, \omega^{\otimes k})
\]
\end{definition}
Concretely, a $p$-adic modular form over $R$ is a rule $f$ which assigns to each triple $(R', E, \alpha)$ consisting of a $p$-complete $R$-algebra, a generalized ordinary elliptic curve $E$ over $R'$, and a level $\Gamma$ structure $\alpha$, an element $f(R', E, \alpha) \in H^0(E, \omega^{\otimes k})$. Furthermore, this assignment only depends on the isomorphism class of $(E, \alpha)$ and commutes with extension of scalars. 

By Lubin-Tate theory, $\mathscr{M}^{<2}_{FG}$ admits a pro-Galois cover by $\textup{Spf}(\mathbb{Z}_p)$,
\[
\textup{Spf}(\mathbb{Z}_p) \longrightarrow \mathscr{M}^{<2}_{FG}
\]
classifying the multiplicative group over $\mathbb{Z}_p$. In particular, this is a torsor for the group $\mathbb{Z}^{\times}_p$, which means that the shearing map
\[
\textup{Spf}(C(\mathbb{Z}_p^{\times}, \mathbb{Z}_p)) \times_{\textup{Spf}(\mathbb{Z}_p)} \textup{Spf}(\mathbb{Z}_p) \longrightarrow \textup{Spf}(\mathbb{Z}_p)\times_{\mathscr{M}^{<2}_{FG}}\textup{Spf}(\mathbb{Z}_p)
\]
is an isomorphism. The subgroups $1+p^n\mathbb{Z}_p \subset \mathbb{Z}^{\times}_p$ give intermediate covers
\[
\mathscr{M}_{FG}(p^n) \longrightarrow \mathscr{M}^{<2}_{FG}
\]
which are Galois with Galois group $(\mathbb{Z}/p^n)^{\times}$. Pulling back this tower gives a diagram
\[
\xymatrix{
\mathscr{M}^{\textup{triv}}_{\Gamma} \ar[r] \ar[d]& \textup{Spf}(\mathbb{Z}_p) \ar[d]\\
\vdots \ar[d]& \vdots\ar[d]\\
\mathscr{M}_{\Gamma}(p^n) \ar[r] \ar[d]& \mathscr{M}_{FG}(p^n) \ar[d]\\
\vdots \ar[d]& \vdots\ar[d]\\
\mathscr{M}^{\textup{ord}}_{\Gamma} \ar[r] & \mathscr{M}^{<2}_{FG}
}
\]
The left hand tower is known as the \emph{Igusa tower}. The formal stack $\mathscr{M}_{\Gamma}(p^n)$ classifies triples $(E, \alpha, \eta)$ where $(E, \alpha)$ is as before and $\eta: \mu_{p^n} \hookrightarrow E[p^n]$ is an injection of $p$-divisible groups. The formal stack $\mathscr{M}^{\textup{triv}}_{\Gamma}$ parameterizes triples $(E, \alpha, \eta)$ where $\eta: \widehat{\mathbb{G}}_m \stackrel{\cong}{\longrightarrow} \widehat{E}$ is a trivialization of the formal group. We record a few facts about this tower.
\begin{proposition} The stack $\mathscr{M}_{\Gamma}(p^n)$ is formally affine for $p$ odd and $n\ge 1$ and $p=2$ for $n \ge 2$. We denote the corresponding $p$-complete rings by $V_n(\Gamma)$. In particular $\mathscr{M}^{\textup{triv}}_{\Gamma}$ is formally affine with coordinate ring $V_{\infty}(\Gamma)$ which is flat over $\mathbb{Z}_p$. 
\end{proposition}
\begin{proof} A proof of representability at level 1 (which implies representability for higher levels) can be found in \cite[5.2]{Beh}. Flatness follows from the same result for $\mathscr{M}^{\textup{ord}}_{\Gamma}$. 
\end{proof}
\begin{remark} Of course, for sufficiently large $N$ and small $\Gamma$, we get representability from the start. 
\end{remark}
The ring $V_\infty(\Gamma)$ is called the ring of {\it generalized $p$-adic modular forms}. This carries an action of $\mathbb{Z}_p^{\times}$ by automorphisms of $\mathbb{G}_m$, as well as a ring map $\psi^p: V_\infty(\Gamma) \longrightarrow V_{\infty}(\Gamma)$ which is a lift of Frobenius. Its effect on $q$-expansions is via
\[
\psi^p: \sum a_nq^n \mapsto \sum a_nq^{pn}
\]
We extend this to a map on $V_{\infty}(\Gamma)_*:= H^0(\mathscr{M}^{\textup{triv}}_{\Gamma}, \omega^{\otimes \ast})$  by asking that $\psi^p(v) = pv$ for the periodicity generator. For any generalized $p$-adic modular form $g_k \in H^0(\mathscr{M}^{\textup{triv}}_{\Gamma}, \omega^{\otimes k})$ we set the following notation:
\[
g_k^{(p)} := \left(1 - \frac{1}{p}\psi^p\right)g_k
\]
so that, on $q$-expansions, we have
\[
g_k^{(p)}(q) = \sum (a_n - p^{k-1}a_{n/p}) q^n, \textup{ where } g_k = \sum a_nq^n
\]

We will also need to know that there is a formal scheme $\textup{Spf}(V_\infty(\Gamma)^{\textup{ss}})$ which captures the data of the ordinary part of a formal neighborhood of the supersingular locus. We will construct this for full level $N$ structure, $\Gamma = \Gamma(N)$, following \cite{Beh} and in general define 
\[
V_{\infty}(\Gamma)^{\textup{ss}} := \left(V_{\infty}(\Gamma(N))^{\textup{ss}}\right)^{\Gamma}.
\]
At level $N$, a formal neighborhood of the supersingular locus is given by 
\[
\mathscr{M}^{\textup{ss}}_{\textup{ell}}(\Gamma(N)) = \coprod \textup{Spf}(W(k_i)[\![u_1]\!])
\]
for some collection of finite fields $k_i$ depending on $N$. We denote the global sections by $A_N$ and define
\[
B_N := A_n[u_1^{-1}]^{\wedge}_p, \quad \mathscr{M}^{\textup{ss}}_{\textup{ell}}(N)^{\textup{ord}} :=\textup{Spf}(B_N)
\]
Then $\textup{Spf}(V_{\infty}(\Gamma(N))^{\textup{ss}})$ is defined by the formal pullback
\[
\xymatrix{\textup{Spf}(V_{\infty}(\Gamma(N))^{\textup{ss}}) \ar[r]\ar[d] & \mathscr{M}(p^{\infty}) \ar[d]\\
\mathscr{M}^{\textup{ss}}_{\textup{ell}}(N)^{\textup{ord}} \ar[r] & \mathscr{M}^{\textup{ord}}_{\textup{ell}}}
\]
\subsection{$K(1)$-local topological modular forms}
The intersection of the previous section with topology comes from the following proposition.
\begin{proposition} There are canonical isomorphisms
\[
L_{K(1)}tmf(\Gamma) \cong \mathcal{O}^{\textup{top}}(\mathscr{M}^{\textup{ord}}_{\Gamma})
\]
\[
L_{K(1)}\left(K_p \wedge L_{K(1)}tmf(\Gamma)\right) \cong \mathcal{O}^{\textup{top}}(\mathscr{M}^{\textup{triv}}_{\Gamma})
\]
and the latter induces an isomorphism of $\theta$-algebras
\[
(K_p)^{\wedge}_*(L_{K(1)}tmf(\Gamma)) \cong (K_p)_* \otimes_{\mathbb{Z}_p} V_{\infty}(\Gamma)
\]
where the structure on the right hand side is diagonal and the structure on $V_\infty(\Gamma)$ comes from the action of $\mathbb{Z}^{\times}_p$ and a lift of Frobenius. Similarly, we have
\[
(K_p)^{\wedge}_*(L_{K(1)}L_{K(2)}tmf(\Gamma)) \cong (K_p)_* \otimes_{\mathbb{Z}_p} V_{\infty}(\Gamma)^{\textup{ss}}
\]
\end{proposition}
\begin{proof} In fact, with our current construction of $tmf(\Gamma)$, this is part of the {\it definition}. However, this proposition can be recovered from any reasonable construction. Indeed, the first statement is a combination of the fact that the right hand side, being a homotopy limit of $K(1)$-local ring spectra, is $K(1)$-local, together with the fact that the Adams-Novikov spectral sequence (descent spectral sequence) terminates at a finite stage with a horizontal vanishing line. The second statement can be proved using a variant of the usual `stacky pullback lemma' for formal stacks, once one identifies the given moduli of elliptic curves with the stack presented by the $L$-complete Hopf algebroid $(\pi_*(L_{K(1)}(K_p \wedge L_{K(1)}tmf(\Gamma))), \pi_*(L_{K(1)}(K_p \wedge K_p \wedge L_{K(1)}tmf(\Gamma))))$. For references at level 1, see \cite[7.9, 8.6]{Beh}. 
\end{proof}
We will occasionally make use of the following notation:
\[
tmf(\Gamma; p^n) := \mathcal{O}^{\textup{top}}(\mathscr{M}_{\Gamma}(p^n))
\]
When the corresponding moduli stack is representable this spectrum has an action of $(\mathbb{Z}/p^n)^{\times}$ by $E_\infty$-ring maps and we have equivalences
\[
tmf(\Gamma; p^n)^{h(\mathbb{Z}/p^n)^{\times}} \cong L_{K(1)}tmf(\Gamma)
\]
\[
((K_p)^{\wedge}_*tmf(\Gamma; p^n))^{(\mathbb{Z}/p^n)^{\times}} = (K_p)^{\wedge}_*L_{K(1)}tmf(\Gamma)
\]
Finally, we record two more computations.
\begin{proposition} For any prime $p$, we have isomorphisms of $\theta$-algebras
\[
(K_p)^{\wedge}_*K_p \cong (K^{\wedge}_p)_* \otimes_{\mathbb{Z}_p}C(\mathbb{Z}^{\times}_p, \mathbb{Z}_p)
\]
\[
(K_p)^{\wedge}_*KO_p \cong (K^{\wedge}_p)_* \otimes_{\mathbb{Z}_p} C(\mathbb{Z}^{\times}_p/\{\pm 1\}, \mathbb{Z}_p)
\]
Where the $\theta$-algebra structures on $C(A, \mathbb{Z}_p)$ come from the evident actions of $\mathbb{Z}^{\times}_p$ together with a lift of Frobenius acting on $\mathbb{Z}_p$. 
\end{proposition}
\begin{proof} This goes back to Strickland. A useful reference is \cite{GHMR} or \cite{BH}. Alternatively, one may compute directly the associated (affine) moduli problem.
\end{proof}
Given a $K(n)$-local spectrum $X$ satisfying some mild conditions, the completed homology $(E_n)^{\wedge}_*X$ is an $L$-complete comodule over the $L$-complete Hopf algebroid $((E_n)_*, (E_n)^{\wedge}_*E_n)$. In particular, this is so if $(E_n)^{\wedge}_*X$ is finitely-generated or pro-free. There is a standard method for computing spaces of maps between $K(n)$-local spectra.
\begin{theorem}[Barthel-Heard] Suppose that $(E_n)^{\wedge}_*X$ is pro-free and $(E_n)^{\wedge}_*Y$ is either finitely generated, pro-free, or has bounded $\mathfrak{m}$-torsion where $\mathfrak{m} \subset \pi_0E_n$ is the maximal ideal. Then the $E_2$-term of the $K(n)$-local $E_n$-based Adams-Novikov spectral sequence is
\[
E_2^{s, t} = \widehat{Ext}^{s,t}_{(E_n)^{\wedge}_*E_n}((E_n)^{\wedge}_*X, (E_n)^{\wedge}_*Y)
\] 
This spectral sequence is strongly convergent, with abutment $\pi_{t-s}F(X, L_{K(n)}Y)$. 
\end{theorem}

We will recall the definition of $\widehat{\textup{Ext}}$ below, but first we remark that, in the case $n=1$, we can often interpret the $s=0$ part in a much more pedestrian fashion. If $M$ is an $L$-complete comodule which is pro-free, then it is a \emph{Morava module}. That is, the coaction map
\[
M \longrightarrow (K_p)^{\wedge}_*K_p \widehat{\otimes}_{(K_p)_*} M = C(\mathbb{Z}^{\times}_p,\mathbb{Z}_p) \widehat{\otimes}_{\mathbb{Z}_p} M
\]
defines a continuous semi-linear action of $\mathbb{Z}^{\times}_p$ given explicitly by
\[
\mathbb{Z}_p[[\mathbb{Z}_p^{\times}]] \widehat{\otimes}_{\mathbb{Z}_p} M \longrightarrow \mathbb{Z}_p[[\mathbb{Z}^{\times}_p]] \widehat{\otimes}_{\mathbb{Z}_p}C(\mathbb{Z}^{\times}_p,\mathbb{Z}_p) \widehat{\otimes}_{\mathbb{Z}_p} M\stackrel{ev \otimes 1}{\longrightarrow} M
\]
In this case we have a natural isomorphism
\[
\textup{Hom}^{\textup{cts}}_{\mathbb{Z}^{\times}_p}(M, N) \cong \textup{Hom}_{\widehat{\textup{Comod}}_{(K_p)^{\wedge}_*K_p}}(M, N)
\]
where the left hand side denotes continuous, $\mathbb{Z}_p^{\times}$-equivariant homomorphisms. At this point we may state our first main calculation.
\begin{theorem}\label{theorem:edge map} The following edge maps are isomorphisms:
\[
[KO_p, L_{K(1)}tmf(\Gamma)] \longrightarrow \textup{Hom}^{\textup{cts}}_{\mathbb{Z}_p^{\times}}(C(\mathbb{Z}^{\times}_p/\{\pm 1\}, \mathbb{Z}_p), V_{\infty}(\Gamma))
\]
\[
[KO_p, L_{K(1)}L_{K(2)}tmf(\Gamma)] \longrightarrow \textup{Hom}^{\textup{cts}}_{\mathbb{Z}_p^{\times}}(C(\mathbb{Z}^{\times}_p/\{\pm 1\}, \mathbb{Z}_p), V_{\infty}(\Gamma)^{\textup{ss}})
\]
\[
[L_{K(1)}tmf(\Gamma), L_{K(1)}tmf(\Gamma)] \longrightarrow \textup{Hom}^{\textup{cts}}_{\mathbb{Z}_p^{\times}}(V_{\infty}(\Gamma), V_{\infty}(\Gamma))
\]
\end{theorem}

We show here how to deduce this from the following theorem.

\begin{theorem}\label{theorem:torsor} Let $G$ be a profinite group, and let $R$ be a complete, local Noetherian graded ring with maximal homogeneous ideal generated by a sequence of regular elements. If $A$ is a $G$-torsor and $M$ is pro-free, then
\[
\widehat{\textup{Ext}}^{s}_{G}(M, A) = 0, \quad s>0
\]
where this is computed in the category of $L$-complete $C(G, R)$-comodules. 
\end{theorem}
We defer the proof and a precise definition of torsor in this context until the next section. When $p$ is odd, the theorem is an immediate consequence of the following
\begin{proposition}\label{proposition:oddvanishing} Let $p\ge 3$. Then we have the following vanishing results for $s>0$
\[
\widehat{\textup{Ext}}^s_{\mathbb{Z}_p^{\times}}((K_p)^{\wedge}_*KO_p, (K_p)^{\wedge}_*L_{K(1)}tmf(\Gamma)) = 0
\]
\[
\widehat{\textup{Ext}}^s_{\mathbb{Z}_p^{\times}}((K_p)^{\wedge}_*KO_p, (K_p)^{\wedge}_*L_{K(1)}L_{K(2)}tmf(\Gamma)) = 0
\]
\[
\widehat{\textup{Ext}}^s_{\mathbb{Z}_p^{\times}}((K_p)^{\wedge}_*L_{K(1)}tmf(\Gamma), (K_p)^{\wedge}_*L_{K(1)}tmf(\Gamma)) = 0
\]
\end{proposition}
\begin{proof} Let $X$ denote either $(K_p)^{\wedge}_*L_{K(1)}tmf(\Gamma)$ or $(K_p)^{\wedge}_*L_{K(1)}L_{K(2)}tmf(\Gamma)$. Then it is enough, by Theorem \ref{theorem:torsor}, to show that $X$ is a summand of a $\mathbb{Z}_p^{\times}$-torsor. To do this, it is enough to realize $X$ as the $G$-fixed points of a $\mathbb{Z}_p^{\times}$-torsor where the order of $G$ is prime to $p$. This can always be done by adding level structure (c.f. \cite[Prop. 7.1, Prop. 8.6]{Beh}).  
\end{proof}

The case $p=2$ is more delicate, in this case the relevant spectral sequence does not collapse. We first reduce the calculation of the $E_2$ term to one in group cohomology.
\begin{proposition} For $s>0$ and $p=2$ we have vanishing as in (\ref{proposition:oddvanishing}) if we take $\textup{Ext}$ over the subgroup $1+4\mathbb{Z}_2$.
\end{proposition}
\begin{proof} This follows from (\ref{theorem:torsor}) and \cite[5.2]{Beh}. 
\end{proof}
If $-1$ acts nontrivially on $V_{\infty}(\Gamma)$ (e.g. for $\Gamma_1(N)$), then, in fact, we could have taken the entire group $\mathbb{Z}_2^{\times}$ in the previous proposition. So the only remaining case is when $-1$ acts trivially (e.g. for $\Gamma_0(N)$). We will compute the homotopy type of the mapping space in question by working our way up from $KO$ (at the cusp), to $KO((q))$ (a punctured neighborhood of the cusp), and finally to $L_{K(1)}tmf(\Gamma)$ or $L_{K(1)}L_{K(2)}tmf(\Gamma)$. 

In the next few propositions, we will denote by $E_r(X, Y)$ (resp. $E_r(X)$) the $K(1)$-local Adams-Novikov spectral sequence computing $\pi_*F(X, Y)$ (resp. $\pi_*X$). 

\begin{proposition} We have an isomorphism of spectral sequences
\[
E_r(KO_2, KO_2) = E_r(KO_2) \widehat{\otimes}_{\mathbb{Z}_2} \textup{Hom}^{\textup{cts}}_{\mathbb{Z}_2^{\times}}(C(\mathbb{Z}_2^{\times}/\{\pm 1\}, \mathbb{Z}_2), C(\mathbb{Z}_2^{\times}/\{\pm 1\}, \mathbb{Z}_2))
\] 
\end{proposition}
\begin{proof} There is an action of $KO_2$ on $F(KO_2, KO_2)$, where the module structure map is adjoint to the composite
\[
KO_2 \wedge KO_2 \wedge F(KO_2, KO_2) \longrightarrow KO_2 \wedge F(KO_2, KO_2) \stackrel{ev}{\longrightarrow} KO_2
\] 
which induces a pairing on spectral sequences. The isomorphism claimed in the theorem certainly holds when $r=2$, so we get an isomorphism of spectral sequences if we can show that the elements in bidegree $(0,0)$ are permanent cycles. To see this, recall from, e.g., \cite[9.2]{AHR}, we have
\[
[KO_2, KO_2] \cong  \textup{Hom}^{\textup{cts}}(C(\mathbb{Z}_2^{\times}/\{\pm 1\}, \mathbb{Z}_2), \mathbb{Z}_2)
\]
But we have a natural isomorphism
\[
 \textup{Hom}^{\textup{cts}}_{\mathbb{Z}_2^{\times}}(C(\mathbb{Z}_2^{\times}/\{\pm 1\}, \mathbb{Z}_2), C(\mathbb{Z}_p^{\times}/\{\pm 1\}, \mathbb{Z}_2)) \cong \textup{Hom}^{\textup{cts}}(C(\mathbb{Z}_2^{\times}/\{\pm 1\}, \mathbb{Z}_2), \mathbb{Z}_2)
\]
Indeed, given an equivariant homomorphism $\phi$ from the left hand side, we note that its behavior is determined by knowledge of $\phi(x^{2k})$ for $k\ge 0$. Since $\lambda \ast x^{2k} = \lambda^{2k} x^{2k}$ for $\lambda \in \mathbb{Z}_2^{\times}$, $\phi(x^{2k})$ must have the same property. The only such function is a scalar multiple of $x^{2k}$, and so determined by an element in $\mathbb{Z}_2$. This completes the proof. 
\end{proof}

\begin{corollary} The preceding result holds for Tate $K$ theories, $KO_2[[q]]$ and $KO_2(\!(q)\!)$. That is, we have isomorphisms of spectral sequences
\[
E_r(KO_2, KO_2[[q]]) = E_r(KO_2) \widehat{\otimes}_{\mathbb{Z}_2} \textup{Hom}^{\textup{cts}}_{\mathbb{Z}_2^{\times}}(C(\mathbb{Z}_2^{\times}/\{\pm 1\}, \mathbb{Z}_2), C(\mathbb{Z}_2^{\times}/\{\pm 1\}, \mathbb{Z}_2[[q]]))
\]
\[
E_r(KO_2, KO_2(\!(q)\!)) = E_r(KO_2) \widehat{\otimes}_{\mathbb{Z}_2} \textup{Hom}^{\textup{cts}}_{\mathbb{Z}_2^{\times}}(C(\mathbb{Z}_2^{\times}/\{\pm 1\}, \mathbb{Z}_2), C(\mathbb{Z}_2^{\times}/\{\pm 1\}, \mathbb{Z}_2(\!(q)\!)))
\]
\end{corollary}
\begin{proof} There is no $\lim^1$ issue in the first case because of surjectivity of the maps in the inverse system. The second case follows from localization of the first. 
\end{proof}
\begin{proposition} Suppose $-1$ acts trivially on $V_{\infty}(\Gamma)$. Let $X$ be either $L_{K(1)}tmf(\Gamma)$ or $L_{K(1)}L_{K(2)}tmf(\Gamma)$. Let $M$ denote the $\mathbb{Z}_p$-module
\[
\textup{Hom}^{\textup{cts}}_{\mathbb{Z}_2^{\times}}((K_2)^{\wedge}_*KO_2, (K_2)^{\wedge}_*X)
\]
Then we have an isomorphism of spectral sequences
\[
E_r(KO_2, X) \cong E_r(KO_2) \widehat{\otimes}_{\mathbb{Z}_2} M
\]
In particular, 
\[
\pi_0F(KO_2, X) = M, \quad \pi_1F(KO_2, X) = M/2
\]
\end{proposition}
\begin{proof} First notice that the isomorphism holds when $r=2$. One way to see this is to compute the Ext group via the Cartan-Eilenberg spectral sequence
\[
H_{\textup{cts}}^s((\mathbb{Z}/4)^{\times}, \widehat{\textup{Ext}}_{1+4\mathbb{Z}_2}^m((K_2)^{\wedge}_*KO_2, (K_2)^{\wedge}_*X)) \Rightarrow \widehat{\textup{Ext}}_{\mathbb{Z}_2^{\times}}^{s+m}((K_2)^{\wedge}_*KO_2, (K_2)^{\wedge}_*X)
\]
Since $(K_2)^{\wedge}_*X$ is a torsor for $1+4\mathbb{Z}_2$, this spectral sequence collapses to
\[
H_{\textup{cts}}^s((\mathbb{Z}/4)^{\times}, (K_2)_* \widehat{\otimes}_{\mathbb{Z}_2} M)
\]
and the action of $(\mathbb{Z}/4)^{\times}$ is trivial on $M$ so we get the desired isomorphism for $r=2$.

For either choice of $X$ we have a restriction map
\[
X \longrightarrow KO_2(\!(q^{1/N})\!)
\]
which induces an injection
\[
\textup{Hom}^{\textup{cts}}_{\mathbb{Z}_2^{\times}}(C(\mathbb{Z}_2^{\times}/\{\pm 1\}, \mathbb{Z}_2), (K_2)^{\wedge}_*X)\longrightarrow  \textup{Hom}^{\textup{cts}}_{\mathbb{Z}_2^{\times}}(C(\mathbb{Z}_2^{\times}/\{\pm 1\}, \mathbb{Z}_2), C(\mathbb{Z}_2^{\times}/\{\pm 1\}, \mathbb{Z}_2(\!(q^{1/N})\!)))
\]
This in turn gives an injection
\[
E_2(KO_2, X) \longrightarrow E_2(KO_2, KO_2(\!(q^{1/N})\!))
\]
The first nontrivial differential is a $d_3$, which must be zero on bidegree $(0,0)$ since this is true of the right hand side. This leaves nothing in the column $t-s = -1$, so that, in fact, the elements of bidegree $(0,0)$ are permanent cycles. The proposition follows. 
\end{proof}

For the final function spectrum, we will need the following result, which is folklore at level 1. 
\begin{theorem} Suppose $-1$ acts nontrivially on $V_{\infty}(\Gamma)$. Then, when $p=2$, $L_{K(1)}tmf(\Gamma)$ splits as a (completed) wedge of copies of $KO_2$ (with no suspensions.)
\end{theorem}
\begin{proof} Arguing as in the previous proposition, we have an injection
\[
E_2(L_{K(1)}tmf(\Gamma)) \longrightarrow E_2(KO_2(\!(q^{1/N})\!))
\]
The $E_2(KO_2)$-module structure on the right-hand side preserves the image, and we get
\[
E_r(L_{K(1)}tmf(\Gamma)) \cong E_r(KO_2) \widehat{\otimes}_{\mathbb{Z}_2}V_{\infty}(\Gamma)^{\mathbb{Z}_2^{\times}}
\] 
in particular, on homotopy we have
\[
\pi_*L_{K(1)}tmf(\Gamma) \cong \pi_*KO_2 \widehat{\otimes}_{\mathbb{Z}_2}V_{\infty}(\Gamma)^{\mathbb{Z}_2^{\times}}
\]
so we just need a map from a wedge of $KO_2$'s inducing this isomorphism. Define an equivariant homomorphism
\[
ev_1: C(\mathbb{Z}_2^{\times}/\{\pm 1\}, \mathbb{Z}_2)  \longrightarrow V_{\infty}(\Gamma)
\]
by $f\mapsto f(1)$. Then consider the composite
\[
 C(\mathbb{Z}_2^{\times}/\{\pm 1\}, \mathbb{Z}_2) \widehat{\otimes}_{\mathbb{Z}_2}V_{\infty}(\Gamma)^{\mathbb{Z}_2^{\times}} \stackrel{ev_1 \otimes 1}{\longrightarrow} V_{\infty}(\Gamma) \widehat{\otimes}_{\mathbb{Z}_2} V_{\infty}(\Gamma)^{\mathbb{Z}_2^{\times}} \stackrel{mult}{\longrightarrow} V_{\infty}(\Gamma)
\]
By the part of Theorem \ref{theorem:edge map} proved so far, we get a map
\[
L_{K(1)}\left(KO_2 \wedge (V_{\infty}(\Gamma)^{\mathbb{Z}_2^{\times}})_+\right) \longrightarrow L_{K(1)}tmf(\Gamma)
\]
inducing the desired isomorphism on homotopy groups. 
\end{proof}
\begin{remark} Despite the splitting, $L_{K(1)}tmf(\Gamma)$ is not a $KO^{\wedge}_2$-algebra. 
\end{remark}
\begin{corollary} The edge map 
\[
[L_{K(1)}tmf(\Gamma), L_{K(1)}tmf(\Gamma)] \longrightarrow \textup{Hom}^{\textup{cts}}_{\mathbb{Z}_p^{\times}}(V_{\infty}(\Gamma), V_{\infty}(\Gamma))
\]
is an isomorphism.
\end{corollary}
\begin{proof} We have already seen the result for $KO_2$ as the source, so the result is immediate from the previous theorem by naturality of edge homomorphisms. 
\end{proof}
\subsection{Cohomology of torsors}
We begin by briefly reviewing the definition of a torsor. Let $G$ be a profinite group and $R$ a complete, local Noetherian graded ring with a maximal homogeneous ideal generated by a sequence of regular elements.
\begin{definition} Let $A$ be a pro-free Morava module. Then $A$ is a $G$-torsor if $A$ is a faithfully flat extension of $A^G$ and the natural map
\[
A \widehat{\otimes}_{A^G}A \longrightarrow C(G, A)
\]
given by $(a, a') \mapsto (\phi(g) = ag(a'))$ is an isomorphism of rings. 
\end{definition}

First we reduce Theorem \ref{theorem:torsor} to a statement about group cohomology, as opposed to Ext groups.
\begin{proposition} Let $M$ and $N$ be pro-free Morava modules. Then we have a natural equivalence
\[
\widehat{\textup{Ext}}^*_{C(G, R)}(R, \textup{Hom}^{\textup{cts}}_{R}(M, N)) \cong \widehat{\textup{Ext}}^*_{C(G, R)}(M, N)
\]
\end{proposition}
\begin{proof} The cochain complex computing the right hand side has terms
\[
\textup{Hom}^{\textup{cts}}_R(M, C(G^n, N))
\]
Since $M$ is pro-free, the natural map to $C(G^n, \textup{Hom}^{\textup{cts}}_R(M, N))$ is an isomorphism, and this computes the left hand side. 
\end{proof}
The theorem is now a direct consequence of the following proposition applied to $\textup{Hom}^{\textup{cts}}_R(M, A)$. 

\begin{proposition} Let $A$ be a $G$-torsor and $M$ a complete $A$-module with compatible continuous $G$-action. Then
\[
\widehat{\textup{Ext}}^s_{C(G, R)}(R, M) = 0, \quad s>0
\]
\end{proposition}
\begin{proof} We recall the definition of this Ext group. Let 
\[
\Omega^n(N): = \Gamma^{\boxtimes_R n+1}\boxtimes_{R} M
\]
be the usual cobar resolution, where $\boxtimes$ is the monoidal structure on $L$-complete modules given by $L_0\circ \otimes$ as in \cite{BH}. Then we have
\[
\widehat{\textup{Ext}}^s_{\Gamma}(R, M) := H^s\left(Hom^{\textup{cts}}_{\Gamma}(R, \Omega^*(M))\right)
\]
But, by the coinduction adjunction, we have
\[
Hom^{\textup{cts}}_{\Gamma}(R, \Omega^n(N)) = \Gamma^{\boxtimes_{R}n} \boxtimes_RM
\]
as $R$-modules. We want to show this complex is acyclic, and it is enough to do this after faithfully flat base change, i.e. we need only show that the complex with terms
\[
\Gamma^{\boxtimes_Rn} \boxtimes_RM \boxtimes_{A^G}A
\]
is acyclic. But, by the torsor assumption, this complex is isomorphic to the Amitsur complex
\[
M \boxtimes_{A^G}A \longrightarrow M\boxtimes_{A^G}(A \boxtimes_{A^G}A) \longrightarrow \cdots
\]
which is acyclic by faithfully flat descent. 
\end{proof}

\subsection{Reinterpretation via $p$-adic measures}
Now we'd like a better description of the group
\[
\textup{Hom}^{\textup{cts}}_{\mathbb{Z}_p^{\times}}(C(\mathbb{Z}^{\times}_p/\{\pm1\}, \mathbb{Z}_p), V_\infty(\Gamma)).
\] 
We first treat the non-equivariant case, following \cite{K2}. Let $X$ be a compact, totally disconnected topological space and $R$ a $p$-adically complete ring. We have
\[
C(X,R) = C(X, \mathbb{Z}_p) \widehat{\otimes}_{\mathbb{Z}_p} R
\]
and a \emph{measure} $\mu$ on $X$ with values in $R$ is a (necessarily continuous) $R$-linear map from $C(X,R)$ to $R$. That is
\[
\textup{Meas}(X, R) := \textup{Hom}_{R}(C(X, R), R) = \textup{Hom}_{\mathbb{Z}_p}(C(X, \mathbb{Z}_p), R)
\]
If $U \subset X$ is compact and open then measures on $U$ are in bijection with measures on $X$ supported on $U$. In the case where $X = \mathbb{Z}_p$ we have a very nice description of measures. The space $C(\mathbb{Z}_p, \mathbb{Z}_p)$ is pro-free with a very explicit basis.
\begin{theorem}[Mahler] For each integer $k\ge 0$ denote by $\binom{x}{k}$ the function $\mathbb{Z}_p \longrightarrow \mathbb{Z}_p$ given by 
\[
x \mapsto \begin{cases} 1 & k=0\\
\frac{x(x-1) \cdots (x-(n-1))}{k!} & k>0
\end{cases}
\]
Then $C(\mathbb{Z}_p, \mathbb{Z}_p)$ is pro-free with basis given by these functions.
\end{theorem}
\begin{proof} Let $f: \mathbb{Z}_p \longrightarrow \mathbb{Z}_p$ be arbitrary. Then we can recover $f$ uniquely as
\[
f(x) = \sum_{k \ge 0} (\Delta^kf)(0)\binom{x}{k}
\]
where $\Delta(g)(x) = g(x+1) - g(x)$ as long as this right hand side makes sense.

We need to know that the numbers $(\Delta^kf)(0)$ converge to zero $p$-adically, and this is not obvious. Many proofs may be found in the literature: \cite{Mah} has four different proofs, and \cite{Elk} has a particularly short one.\end{proof}
\begin{corollary} Let $R$ be a $p$-complete $\mathbb{Q}_p$-algebra, then evaluation on the functions $x^k$ for $k\ge m$ gives a bijection
\[
\textup{Meas}(\mathbb{Z}_p, R) \cong \prod_{k\ge 0} R
\]
\end{corollary}
\begin{proof} The functions $x^k$ for $k\ge 0$ form a basis of $C(\mathbb{Z}_p,R)$ in this case since $k! \binom{x}{k} = x^k + O(x^{k-1})$ and $k!$ is a unit.  
\end{proof}
\begin{definition}\label{definition:kummer} Let $R$ be a $p$-adically complete, flat $\mathbb{Z}_p$-algebra and $X$ a compact, totally disconnected subspace of $\mathbb{Z}_p$. We say that a sequence $\{b_k\} \in R[1/p]$ satisfies the \emph{generalized Kummer congruences} (for $X$) if, for every polynomial $\sum a_nz^n \in \mathbb{Q}_p[z] \cap C(X, \mathbb{Z}_p)$ we have
\[
\sum a_nb_n \in R \subset R[1/p]
\]
\end{definition}
\begin{proposition} Let $R$ be a $p$-adically complete, flat $\mathbb{Z}_p$-algebra. Then evaluation on $x^k$ gives an injection
\[
\textup{Meas}(\mathbb{Z}_p, R) \hookrightarrow \prod_{k \ge 0} R[1/p]
\]
onto the set of sequences $\{b_k\}$ satisfying the generalized Kummer congruences.
\end{proposition}
\begin{proof} Injectivity follows from the previous corollary and flatness. To characterize the image, note that we really only need to check the polynomials $\binom{x}{k}$, and this is a necessary and sufficient condition by Mahler's theorem. 
\end{proof}
\begin{corollary}\label{corollary:kummer condition} Let $R$ be as above and fix $m\ge 0$. Evaluation on $x^{2k}$ for $2k\ge m$ gives an injection
\[
\textup{Meas}(\mathbb{Z}^{\times}_p/\{\pm 1\}, R) \hookrightarrow \prod_{k \ge m} R[1/p]
\]
onto the set of sequences $\{b_k\}$ such that $b_{2k-1} = 0$ for all $k$ and satisfying the generalized Kummer congruences.
\end{corollary}
\begin{proof} This is immediate except for the added statement about only needing sufficiently high powers $x^{2k}$. To see this, note that $x^{(p-1)p^r} \rightarrow 1$ uniformly as functions on $\mathbb{Z}^{\times}_p/\{\pm 1\}$. Indeed, any $p$-adic unit $c$ satisfies $c^{p-1} \equiv \pm 1 \textup{ mod }p$ (where the minus only occurs when $p=2$). It follows that $c^{(p-1)p^r} \equiv (\pm 1)^{p^r} \textup{ mod }p^{r+1}$ and taking limits completes the proof. 
\end{proof}

We can now give a purely algebraic description of the group $[KO_p, L_{K(1)}tmf(\Gamma)]$.
\begin{theorem} Evaluation at the generator in $\pi_{4k}KO_p$, followed by division by $2$ when $k$ is odd, gives an injection
\[
[KO_p, L_{K(1)}tmf(\Gamma)] \longrightarrow \prod_{k \ge m} \textup{MF}_{p, k}(\Gamma, \mathbb{Q}_p)
\]
onto the set of sequences $\{g_k\}_{k\ge m}$ of $p$-adic modular forms such that 
\begin{enumerate} 
\item For $k$ odd, $g_k = 0$.
\item The sequence satisfies the generalized Kummer congruences. 
\end{enumerate}
Moreover, given $f: KO_p \longrightarrow L_{K(1)}tmf(\Gamma)$, the effect on $\pi_0$ is given by
\[
\pi_0f(1) = \lim_r g_{(p-1)p^r}
\]
\end{theorem}
\begin{proof} Recall that we have a natural inclusion
\[
MF_{p, k}(\Gamma, \mathbb{Q}_p) \hookrightarrow V_{\infty}(\Gamma, \mathbb{Q}_p)
\]
given by pulling back to a section of $\omega^{\otimes k}$ and then multiplying by a fixed unit in $\omega^{-1}$ coming from a $(p-1)$st root of the Hasse invariant. It suffices to show that the following diagram commutes:
\[
\xymatrix{
[KO_p, L_{K(1)}tmf(\Gamma)] \ar[r]^{\frac{1}{2}ev_{v^{2k}}}\ar[d]_{\cong} & \prod_{k \ge m} MF_{p,k}(\Gamma, \mathbb{Q}_p)\ar[d]\\
\textup{Meas}_{\mathbb{Z}_p^{\times}}(\mathbb{Z}_p^{\times}/\{\pm 1\}, V_{\infty}(\Gamma)) \ar[r]^-{ev_{x^{2k}}} & \prod_{k \ge m} V_{\infty}(\Gamma, \mathbb{Q}_p)
}
\]
Here $v^{2k} \in \pi_{4k}KO_p$ is the image of the similarly named element in $\pi_*bu$ under the forgetful map, and the subscript on measures indicates we restrict to $\mathbb{Z}^{\times}_p$-equivariant homomorphisms $C(\mathbb{Z}^{\times}_p/\{\pm 1\}, \mathbb{Z}_p) \longrightarrow V_{\infty}(\Gamma)$. 

Assuming the commutativity of the diagram, we see that evaluation at $v^{2k}$ gives an injection into the prescribed set of sequences. On the other hand, given such a sequence, we automatically get a measure by the preceding corollary and we need only check it is equivariant. But we may check equivariance on the dense subspace of $C(\mathbb{Z}^{\times}_p/\{\pm 1\}, \mathbb{Z}_p)$ consisting of the even polynomials. In this case equivariance is equivalent to the requirement that $x^{2k}$ maps to an element $g$ of $V_{\infty}(\Gamma)$ which satisfies
\[
\lambda \ast g = \lambda^{2k}g, \quad \lambda \in \mathbb{Z}^{\times}_p
\] 
Since $p$-adic modular forms of weight $2k$ satisfy this requirement, we conclude that the measure is equivariant. 

It remains to check the commutativity of the diagram. This follows from the commutativity of the diagram
\[
\xymatrix{
[K_p, L_{K(1)}tmf(\Gamma)] \ar[r]^{ev_{v^{2k}}}\ar[d]_{\cong} & \prod_{k \ge m} MF_{p,k}(\Gamma, \mathbb{Q}_p)\ar[d]\\
\textup{Meas}_{\mathbb{Z}_p^{\times}}(\mathbb{Z}_p^{\times}, V_{\infty}(\Gamma)) \ar[r]^-{ev_{x^{2k}}} & \prod_{k \ge m} V_{\infty}(\Gamma, \mathbb{Q}_p)
}
\]
by naturality and the observation that the map $KO \longrightarrow K$ sends $v^{2k} \in \pi_*KO_p$, as we've been denoting it, to $2v^{2k} \in \pi_*K_p$. The commutativity of this simpler diagram follows from the same proof as in \cite[9.5]{AHR}. 
\end{proof}
\subsection{$K(1)$-local orientations}
Before completing the proof of the main theorem of this section, we need to review the formula for Rezk's logarithm in the $K(1)$-local case. Let $R$ be a $K(1)$-local $E_\infty$ ring spectrum and denote by $\psi^p$ the power operation on $\pi_*R$ lifting the Frobenius. 
\begin{theorem}[Rezk]\label{theorem:formula} There is a $K(1)$-local equivalence
\[
\ell_1: gl_1R \longrightarrow R
\]
whose effect on homotopy groups in positive degrees is given by
\[
1+x \mapsto \left(1- \frac{1}{p}\psi^p\right)x
\]
\end{theorem}
\begin{corollary}\label{corollary:formula} The effect of $\ell_1$ on $\pi_{2k}L_{K(1)}gl_1tmf(\Gamma) \otimes_{\mathbb{Z}_p} \mathbb{Q}_p \cong M_{p,k}(\Gamma, \mathbb{Q}_p)$ for $k\ge 1$ is given by
\[
g \mapsto g^{(p)}
\]
\end{corollary}
We will also need the following result
\begin{theorem}[Ando-Hopkins-Rezk] \label{theorem:diagram} If $c \in \mathbb{Z}_p^{\times}/\{\pm 1\}$ is a topological generator, we have a commutative diagram
\[
\xymatrix{L_{K(1)}S \ar[r]^{\rho(c)} & KO_p\ar[r]^{1-\psi^c} \ar[d]^{\cong}_{b_c}& KO_p\ar@{=}[d]\\
L_{K(1)}gl_1S \ar[r]\ar[u]^{\ell_1}_{\cong} & L_{K(1)}gl_1S/spin \ar[r] & KO_p}
\]
where $\rho(c)^{-1} = \frac{1}{2p}\log(c^{p-1})$. 
\end{theorem}
\begin{proof} First we construct the map $b_c$. Recall that the (proven) Adams Conjecture states that the composite
\[
\xymatrix{BO^{\wedge}_p \ar[r]^{1-\psi^c} & BO^{\wedge}_p \ar[r]^{J} & (BGL_1S)^{\wedge}_p}
\]
is nullhomotopic. This remains true upon taking connective covers, so that we have a nullhomotopy for the composite
\[
\xymatrix{BSpin^{\wedge}_p \ar[r]^{1-\psi^c} & BSpin^{\wedge}_p \ar[r]^J& (BGL_1S)^{\wedge}_p}
\]
This produces a homotopy commutative diagram
\[
\xymatrix{
\textup{fib}(1-\psi^c) \ar[r]\ar[d]^{A_c} & BSpin^{\wedge}_p \ar[r] \ar[d]^{B_c} & BSpin^{\wedge}_p \ar@{=}[d]\\
(GL_1S)^{\wedge}_p \ar[r] & (GL_1S/Spin)^{\wedge}_p \ar[r] & BSpin^{\wedge}_p \ar[r]^J & (BGL_1S)^{\wedge}_p
}
\]
On the other hand, for any $n$ we have a factorization
\[
\xymatrix{
\textup{Spectra} \ar[rr]^{L_{K(n)}}\ar[dr]_{\Omega^{\infty}} && \textup{Spectra}_{K(n)}\\
&\textup{Spaces}\ar[ur]_{\Phi_n}
}
\]
where $\Phi_n$ is the Bousfield-Kuhn functor. So, using $\Phi_1$, we get a homotopy commutative diagram
\[
\xymatrix{
L_{K(1)}\textup{fib}(1-\psi^c) \ar[r]\ar[d]^{a_c} & L_{K(1)}bspin \ar[r]^{1-\psi^c} \ar[d]^{b_c} & L_{K(1)}bspin \ar@{=}[d]\\
L_{K(1)}gl_1S \ar[r] & L_{K(1)}gl_1S/spin \ar[r] & L_{K(1)}bspin
}
\]
We have a canonical equivalence
\[
L_{K(1)}bspin \cong KO_p
\] 
so our diagram becomes
\[
\xymatrix{
j_c \ar[r]\ar[d]^{a_c} & KO_p \ar[r]^{1-\psi^c} \ar[d]^{b_c} & L_{K(1)}bspin \ar@{=}[d]\\
L_{K(1)}gl_1S \ar[r] & L_{K(1)}gl_1S/spin \ar[r] &KO_p
}
\]
Since $c$ is a generator, the computation of the $K(1)$-local sphere shows that the map $a_c$ is an equivalence, whence so is $b_c$. Now {\it define} the map $\rho(c): L_{K(1)}S \longrightarrow KO_p$ to make the following diagram commute
\[
\xymatrix{L_{K(1)}S \ar[r]^{\rho(c)} & KO_p\ar[r]^{1-\psi^c} \ar[d]^{\cong}_{b_c}& KO_p\ar@{=}[d]\\
L_{K(1)}gl_1S \ar[r]\ar[u]^{\ell_1}_{\cong} & L_{K(1)}gl_1S/spin \ar[r] & KO_p}
\]
The description of $\rho(c)$ as in the theorem statement can be found as \cite[7.15]{AHR}.  
\end{proof}
\begin{proof}[Proof of Theorem \ref{theorem:K(1)-local parameterization}] This is exactly as in \cite[14.6]{AHR}. For convenience, we recall the argument here.
We are interested in maps $\alpha$ making the following diagram commute up to homotopy
\[
\xymatrix{
gl_1S \ar[dr]\ar[r] & gl_1S/spin\ar@{-->}[d]^{\alpha}\\
& L_{K(1)}gl_1tmf(\Gamma)^{\wedge}_p
}
\]
(We are permitted to consider just homotopy classes here by the argument in \cite[14.3]{AHR}). Since the target is $K(1)$-local, we may use the previous theorem to replace this diagram with the following one
\[
\xymatrix{
L_{K(1)}S \ar[r]^{\rho(c)} \ar[dr]& KO_p \ar@{-->}[d]^{\ell_1\alpha b_c}\\
& L_{K(1)}tmf(\Gamma)^{\wedge}_p
}
\]
By Theorem \ref{theorem:edge map} a map $KO_p \longrightarrow L_{K(1)}tmf(\Gamma)^{\wedge}_p$ is determined by its rationalization so we must understand the composite
\[
\xymatrix{
& KO_p \ar@{-->}[d]^{\ell_1\alpha b_c}&\\
& L_{K(1)}tmf(\Gamma)^{\wedge}_p \ar[r] &L_{K(1)}tmf(\Gamma)^{\wedge}_p \otimes \mathbb{Q}
}
\]  
Since $gl_1S \otimes \mathbb{Q}$ is contractible, Theorem \ref{theorem:diagram} implies that we have a factorization
\[
\xymatrix{
& KO_p \ar@{-->}[d]^{\ell_1\alpha b_c}\ar[r]^{1-\psi^c}&KO_p\ar@{-->}[d]^{\ell_1\beta}\\
& L_{K(1)}tmf(\Gamma)^{\wedge}_p \ar[r] &L_{K(1)}tmf(\Gamma)^{\wedge}_p \otimes \mathbb{Q}
}
\]
for some $\beta: KO_p \longrightarrow L_{K(1)}gl_1tmf(\Gamma)^{\wedge}_p \otimes \mathbb{Q}$. Thus, if $t_k(\alpha)$ are the moments of the measure determined by $\ell_1\alpha b_c$, and $b_k(\alpha): = \beta_*\nu^{k} \in MF_{p,k}(\Gamma) \otimes \mathbb{Q}$, then the diagram and the formula (\ref{corollary:formula}) implies
\[
t_k(\alpha) = (1-c^k)b_k^{(p)}
\]
Thus, the existence of such a commutative square is determined by a sequence of $p$-adic modular forms $\{g_k\}_{k \ge 4}$ such that the sequence $\{(1-c^k)g^{(p)}\}$ extends to a measure $\mu_c$ on $\mathbb{Z}_p^{\times}/\{\pm 1\}$ valued in $V_{\infty}(\Gamma)$. This square, in turn, makes the diagram
\[
\xymatrix{
L_{K(1)}S \ar[r]^{\rho(c)}\ar[dr]& KO_p \ar@{-->}[d]^{\ell_1\alpha b_c}\ar[r]^{1-\psi^c}&KO_p\ar@{-->}[d]^{\ell_1\beta}\\
& L_{K(1)}tmf(\Gamma)^{\wedge}_p \ar[r] &L_{K(1)}tmf(\Gamma)^{\wedge}_p \otimes \mathbb{Q}
}
\]
commute if and only if the effect of $\ell_1\alpha b_c$ on $\pi_0$ is $\rho(c)^{-1}$. By \cite[7.15]{AHR} this happens if and only if $\int_{\mathbb{Z}_p^{\times}/\{\pm 1\}} 1 d\mu_c = \frac{1}{2p} \log(c^{p-1})$, which completes the proof. 
\end{proof}
\section{Gluing and the Atkin operator}
Crucial to the construction of these genera is an understanding of the bottom map in the fiber square
\[
\xymatrix{
L_{K(1)\vee K(2)}gl_1tmf(\Gamma)^{\wedge}_p \ar[r]\ar[d] & L_{K(2)}gl_1tmf(\Gamma)^{\wedge}_p\ar[d] \\
 L_{K(1)}gl_1tmf(\Gamma)^{\wedge}_p\ar[r]& L_{K(1)}L_{K(2)}gl_1tmf(\Gamma)^{\wedge}_p
}
\]
Using a $K(2)$-local version of Rezk's logarithm, and the $K(1)$-local version we've already met, we can replace this square with the following one
\[
\xymatrix{
L_{K(1)\vee K(2)}gl_1tmf(\Gamma)^{\wedge}_p \ar[r]\ar[d] & L_{K(2)}tmf(\Gamma)^{\wedge}_p\ar[d] \\
 L_{K(1)}tmf(\Gamma)^{\wedge}_p\ar[r]& L_{K(1)}L_{K(2)}tmf(\Gamma)^{\wedge}_p
}
\]
By naturality, the right hand vertical map is the usual one, but the bottom horizontal map is {\it not}. One of the key insights of Ando-Hopkins-Rezk is a computation of this map in terms of the Atkin operator on $p$-adic modular forms. In \S3.1 we review the necessary algebraic facts about this operator and its relationship to Hecke operators. In \S3.2 we construct a topological lift of the Atkin operator on $K(1)$-local topological modular forms and show that it fits in a fiber square as above. Finally, in \S3.3
we compute the connected components of $\textup{Nulls}(spin, L_{K(1)\vee K(2)}gl_1tmf(\Gamma)^{\wedge}_p)$, which is the penultimate step in the program to understanding $p$-complete orientations. 
\subsection{Hecke operators on modular forms}
Fix $\Gamma \in \{\Gamma_1(N), \Gamma_0(N)\}$, and a prime $p$ not dividing $N$. A $(\Gamma_1(N); p)$ level structure on a generalized elliptic curve $E_{/S}$ is a pair $(P, C)$ where $P$ is a $\Gamma_1(N)$-level structure and $C$ is a finite locally free $S$-subgroup scheme in $E^{\textup{sm}}$ that is cyclic of order $p$ subject to the condition that the effective Cartier divisor
\[
\sum_{j \in \mathbb{Z}/N\mathbb{Z}} (jP+C)
\]
meets all irreducible components of all geometric fibers. A $(\Gamma_0(N); p)$ level structure is a pair $(G, C)$ consisting of a cyclic subgroup $G \subset E^{\textup{sm}}$ of order $N$ such that, $fppf$ locally where $G$ admits a generator $P$, the pair $(P,C)$ is a $(\Gamma_1(N); p)$ level structure. These define stacks $\mathscr{M}_{(\Gamma_1(N);p)}$ and $\mathscr{M}_{(\Gamma_0(N);p)}$. On the smooth loci, there are two canonical maps $\pi_1^0, \pi_2^0: \mathscr{M}^0_{(\Gamma;p)} \longrightarrow \mathscr{M}^0_{(\Gamma;p)}$ given for $(\Gamma_1(N);p)$ by
\[
\pi_1^0(E; P, C) = (E,P), \quad \pi_2^0(E; P, C) = (E/C, P \textup{mod }C)
\]
and defined via flat descent for $(\Gamma_0(N);p)$. 
\begin{theorem}[Conrad] The correspondence $(\pi_1^0, \pi_2^0)$ uniquely extends to a finite, flat correspondence, the \emph{Hecke correspondence},
\[
\xymatrix{
&\mathscr{M}_{(\Gamma;p)}\ar[dr]^{\pi_2}\ar[dl]_{\pi_1}&\\
\mathscr{M}_{\Gamma}&&\mathscr{M}_{\Gamma}
}
\]
Likewise, the natural map on the smooth locus
\[
(\pi_2^0)^*\omega_{\Gamma} \longrightarrow \omega_{(\Gamma;p)}
\]
uniquely extends to a map on all of $\mathscr{M}_{(\Gamma;p)}$.
\end{theorem}

Completing at $p$ and restricting to the ordinary locus we get a correspondence
\[
\xymatrix{
& \mathscr{M}^{\textup{ord}}_{(\Gamma;p)} \ar[dl]_{\pi_1}\ar[dr]^{\pi_2}&\\
\mathscr{M}^{\textup{ord}}_{\Gamma}&&\mathscr{M}^{\textup{ord}}_{\Gamma}
}
\]
which induces a map $pT_p: M_{p,k}(\Gamma) \longrightarrow M_{p,k}(\Gamma)$ defined as the composite
\[
\xymatrix{
H^0(\mathscr{M}^{\textup{ord}}_{\Gamma}, \omega_{\Gamma}^{\otimes k})\ar[r]^{\pi_2^*} &H^0(\mathscr{M}^{\textup{ord}}_{(\Gamma;p)}, \pi_2^*\omega_{\Gamma}^{\otimes k})\ar[r]& H^0(\mathscr{M}^{\textup{ord}}_{(\Gamma;p)}, \omega^{\otimes k})=  H^0(\mathscr{M}^{\textup{ord}}_{\Gamma}, (\pi_1)_*\pi_1^*\omega^{\otimes k}) \ar[r] & H^0(\mathscr{M}^{\textup{ord}}_{\Gamma}, \omega^{\otimes k})
}
\]
The image of this map lies in $pM_{p,k}$, which is torsion free, so we have a well-defined operator
\[
T_p: M_{p,k}(\Gamma) \longrightarrow M_{p,k}(\Gamma)
\]
called the $p$th Hecke operator. As it happens, $\mathscr{M}^{\textup{ord}}_{(\Gamma;p)}$ is the disjoint union of two copies of $\mathscr{M}^{\textup{ord}}_{\Gamma}$. The correspondence then decomposes as a sum of two correspondences corresponding to the following finite, flat maps:
\[
\textup{Frob}: (E; P,C) \mapsto (E^{(p)}, \phi(P), \phi(C))
\]
\[
\langle p\rangle V: (E^{(p)}; p\phi(P), \phi(C)) \mapsto (E; P, C)
\]
Define the {\it Atkin operator} $U_p$ as $\frac{1}{p}\textup{tr}(\textup{Frob})$ on cohomology, and similarly define $\psi^p$ as the trace of the other map divided by $p$. We get
\[
pT_p = \textup{tr}(\textup{Frob}) + \psi^p
\]
on $M_{p,k}(\Gamma)$, which yields a $p$-adic lift of the Eichler-Shimura relation:
\[
T_p = U_p + \frac{1}{p}\psi^p
\] 

By pulling back along the map from $\mathscr{M}^{\textup{triv}}_{\Gamma}$, all of these operators are defined on $V_{\infty}(\Gamma)$ as well, and the same $p$-adic lift of the Eichler-Shimura relation holds. 

The effect on $q$-expansions of these operators is as follows for $f \in M_k(\Gamma_0(N))$ with $f(q) = \sum a_nq^n$:
\[
U_pf(q) = \sum a_{np} q^n, \quad \psi^pf(q) = p^k\sum a_nq^{np}
\]
and for level $\Gamma_1(N)$ there are similar formulae that depend on the nebentypus of the modular form. From these, or directly from the definition, one verifies that $U_p\psi^p = p^k$ for arbitrary level. 
\begin{remark} The reader may be confused that $\psi^p$ seems to be acting simultaneously like the Frobenius and a multiple of the Verschiebung. The mystery is solved by noting that there are two distinct notions of a level $p$ structure on an elliptic curve. One is a morphism $\mu_p \hookrightarrow E[p]$ (an Igusa level structure) and one is a morphism $\mathbb{Z}/p\mathbb{Z} \longrightarrow E[p]$ (a classical level $p$ structure). These are `transposes' of each other in the same way that the Frobenius and the Verschiebung are transposes of each other. See \cite[3.12]{Gro} and the remarks and references therein. 
\end{remark}
\begin{remark} In \cite{AHR}, for level 1 modular forms, the authors use the notation $V_p$ for what corresponds to our $\frac{1}{p^k}\psi^p$ in weight $k$. 
\end{remark}
\subsection{Lifting the Atkin operator}
In addition to Rezk's logarithm for $K(1)$-local $E_\infty$ rings, he also constructed a $K(n)$-local version specifically for forms of Morava $E$-theory. For our purposes, we require the following statement:
\begin{theorem}[Rezk] Let $k$ be a perfect field and $G$ a height $n$ formal group, denote by $E(k, G)$ the associated Morava $E$-theory. Then there is a $K(n)$-local equivalence
\[
\ell_n: gl_1E(k, G) \longrightarrow E(k, G)
\] 
Moreover, this is natural in the pair $(k, G)$. 

At height $2$, in the case $E(k, \widehat{C})$ for $C$ a supersingular elliptic curve over $k$, the effect on $\pi_{2k}$ is: 
\[
1+f \mapsto (1-T_p + p^{k-1})f
\]
\end{theorem}
We would like to apply this map to the $K(2)$-localization of $tmf(\Gamma)$. The relationship comes from the following
\begin{proposition}\label{proposition:supersingular} Let $\textup{SS}_p(\Gamma)$ denote the set of suspersingular elliptic curves at the prime $p$ with level $\Gamma$-structure. Then we have an equivalence
\[
L_{K(2)}tmf(\Gamma) \cong \prod_{(C,\alpha) \in \textup{SS}_p} E(k_C, \widehat{C})^{h\textup{Aut}(C,\alpha)}
\]
\end{proposition}
\begin{proof} On an \'etale map from an affine, this is the definition of $L_{K(2)}\mathcal{O}^{\textup{top}}$ (there are no automorphisms left). So choose an affine cover, and then descent gives the desired homotopy fixed points. 
\end{proof}
We can tie these two together to get a logarithm for topological modular forms.
\begin{proposition} There is a $K(2)$-local equivalence
\[
gl_1tmf(\Gamma)^{\wedge}_p \longrightarrow L_{K(2)}tmf(\Gamma)^{\wedge}_p
\]
such that the following diagram commutes for $k\ge 1$,
\[
\xymatrix{
\pi_{2k}gl_1tmf(\Gamma)^{\wedge}_p \ar[r]\ar[d]& \pi_{2k}L_{K(2)}tmf(\Gamma)^{\wedge}_p\ar[d]\\
MF_k(\Gamma)^{\wedge}_p \ar[r]_{1-T_p+p^{k-1}} & MF_{p,k}(\Gamma)
}
\]
\end{proposition}  
\begin{proof} Recall that, by \cite{Bou2, Kuh}
, we have
\[
L_{K(n)}gl_1R \cong L_{K(n)}gl_1L_{K(n)}R
\]
so we may replace $tmf(\Gamma)^{\wedge}_p$ by its $K(2)$-localization in the source. By Goerss-Hopkins, there is an action of $\mathbb{G}_n$ on each $E(k, G)$ in the homotopy theory of $E_\infty$-rings. Applying $gl_1$ and the natural transformation $\ell_2: gl_1 \longrightarrow L_{K(2)}$ we get a map of diagrams. Now take the homotopy limits indicated in Proposition \ref{proposition:supersingular} and use that $gl_1$ commutes with homotopy limits, as it is a right adjoint.
\end{proof}

Using this equivalence we get the following diagram:
\[
\xymatrix{
L_{K(1)\vee K(2)}gl_1tmf(\Gamma) \ar[r] \ar[d] & L_{K(2)}gl_1tmf(\Gamma)\ar[r]^{\ell_2}_{\cong} \ar[d]&L_{K(2)}tmf(\Gamma)\ar[d]\\
L_{K(1)}gl_1tmf(\Gamma)\ar[d]^{\ell_1}_{\cong} \ar[r]&L_{K(1)}L_{K(2)}gl_1tmf(\Gamma) \ar[r]^{L_{K(1)}\ell_2}_{\cong}&L_{K(1)}L_{K(2)}tmf(\Gamma) \\
L_{K(1)}tmf(\Gamma)\ar@{-->}[urr]  &&
}
\]
The remainder of the section will be devoted to describing this dotted arrow as in the statement below. 
\begin{theorem}\label{theorem:atkin} There is a map $U_p: L_{K(1)}tmf(\Gamma) \longrightarrow L_{K(1)}tmf(\Gamma)$ making the following two diagrams commute
\[
\xymatrix{
\pi_{2k}L_{K(1)}tmf(\Gamma)^{\wedge}_p\ar[d]\ar[r]^{U_p} & \pi_{2k}L_{K(1)}tmf(\Gamma)^{\wedge}_p\ar[d]\\
MF_{p,k}(\Gamma) \ar[r]_{U_p} & MF_{p,k}(\Gamma)
}
\]
\[
\xymatrix{
L_{K(1)\vee K(2)}gl_1tmf(\Gamma)^{\wedge}_p \ar[rr]\ar[d] & & L_{K(2)}tmf(\Gamma)^{\wedge}_p\ar[d]\\
L_{K(1)}tmf(\Gamma)^{\wedge}_p \ar[r]_{1-U_p} & L_{K(1)}tmf(\Gamma)^{\wedge}_p \ar[r] & L_{K(1)}L_{K(2)}tmf(\Gamma)^{\wedge}_p
}
\]
\end{theorem}
\begin{remark} This theorem is claimed in \cite{AHR}. See also \cite{Ba1}. 
\end{remark}
Before turning to the proof, we note an interesting and as-yet unexplained corollary, which is the existence of a logarithm for $p$-complete $tmf$ (though it is no longer an equivalence).
\begin{corollary} \label{corollary:tmf log}There is a fiber square
\[
\xymatrix{
L_{K(1)\vee K(2)}gl_1tmf(\Gamma)^{\wedge}_p \ar[r]^-{\ell_{tmf}}\ar[d] & Tmf(\Gamma)^{\wedge}_p\ar[d] \\
L_{K(1)}tmf(\Gamma)^{\wedge}_p \ar[r]_{1-U_p} & L_{K(1)}tmf(\Gamma)^{\wedge}_p
}
\]
\end{corollary}
\begin{proof} This follows from pasting homotopy pullback squares together with the fact that $L_{K(1)\vee K(2)}tmf(\Gamma)^{\wedge}_p = Tmf(\Gamma)^{\wedge}_p$. 
\end{proof}
First we construct the map.
\begin{lemma} There is a unique map
\[
U_p: L_{K(1)}tmf(\Gamma) \longrightarrow L_{K(1)}tmf(\Gamma)
\]
whose effect upon applying $(K_p)^{\wedge}_*$ is the Atkin operator for generalized $p$-adic modular forms.
\end{lemma}
\begin{proof} This is immediate from Theorem \ref{theorem:edge map} and the fact that $U_p$ is $\mathbb{Z}_p^{\times}$ equivariant. 
\end{proof}
\begin{lemma} We have a commutative diagram
\[
\xymatrix{
\pi_{2k}L_{K(1)}tmf(\Gamma)^{\wedge}_p\ar[d]\ar[r]^{U_p} & \pi_{2k}L_{K(1)}tmf(\Gamma)^{\wedge}_p\ar[d]\\
MF_{p,k}(\Gamma) \ar[r]_{U_p} & MF_{p,k}(\Gamma)
}
\]
\end{lemma}
\begin{proof} This follows from naturality of the edge homomorphism in the Adams-Novikov spectral sequence
\end{proof}
This proves half of Theorem \ref{theorem:atkin}. For the other half, we first need a corollary of Theorem \ref{theorem:edge map}. 

\begin{corollary}\label{corollary:check rational} Evaluation on $\pi_*$ gives an injection
\[
[L_{K(1)}tmf(\Gamma)^{\wedge}_p, L_{K(1)}L_{K(2)}tmf(\Gamma)^{\wedge}_p] \longrightarrow \textup{Hom}(\pi_*L_{K(1)}tmf(\Gamma)^{\wedge}_p \otimes \mathbb{Q}, \pi_*L_{K(1)}L_{K(2)}tmf(\Gamma)^{\wedge}_p \otimes \mathbb{Q})
\]
\end{corollary}
\begin{proof} By (\ref{theorem:edge map}), the left hand side is torsion-free.
\end{proof}
\begin{proposition} The diagram 
\[
\xymatrix{
L_{K(1)\vee K(2)}gl_1tmf(\Gamma)^{\wedge}_p \ar[rr]\ar[d] & & L_{K(2)}tmf(\Gamma)^{\wedge}_p\ar[d]\\
L_{K(1)}tmf(\Gamma)^{\wedge}_p \ar[r]_{1-U_p} & L_{K(1)}tmf(\Gamma)^{\wedge}_p \ar[r] & L_{K(1)}L_{K(2)}tmf(\Gamma)^{\wedge}_p
}
\]
commutes.
\end{proposition}
\begin{proof} It suffices to check commutativity of the following diagram
\[
\xymatrix{
L_{K(1)}gl_1tmf(\Gamma) \ar[r]\ar[d]^{\ell_1} & L_{K(1)}L_{K(2)}gl_1tmf(\Gamma) \ar[d]^{L_{K(1)}\ell_2}\\
L_{K(1)}tmf(\Gamma) \ar[r]^-{1-U_p} & L_{K(1)}L_{K(2)}tmf(\Gamma)
}
\]
By Corollary \ref{corollary:check rational} we need only check that the diagram
\[
\xymatrix{
\pi_*L_{K(1)}gl_1tmf(\Gamma)\otimes \mathbb{Q} \ar[r]\ar[d]^{\ell_1} & L_{K(1)}L_{K(2)}gl_1tmf(\Gamma)\otimes \mathbb{Q} \ar[d]^{L_{K(1)}\ell_2}\\
\pi_*L_{K(1)}tmf(\Gamma)\otimes\mathbb{Q} \ar[r]^-{1-U_p} & \pi_*L_{K(1)}L_{K(2)}tmf(\Gamma)\otimes \mathbb{Q}
}
\]
but now this is a statement about modular forms, which follows from Rezk's formulae for the logarithms and the $p$-adic lift of the Eichler-Shimura relation. 
\end{proof}
\subsection{Spin orientations in the unscrewable case}
We are now ready to solve the extension problem
\[
\xymatrix{
spin \ar[r] & gl_1S \ar[r]\ar[d] & gl_1S/spin \ar@{-->}[dl]\\
& L_{K(1)\vee K(2)}gl_1tmf(\Gamma)^{\wedge}_p
}
\]
\begin{theorem}\label{theorem:unscrewable orientations} If $p\ne 2$, the map
\[
\pi_0\textup{Nulls}(spin, L_{K(1)\vee K(2)}gl_1tmf(\Gamma)^{\wedge}_p) \longrightarrow \pi_0\textup{Nulls}(spin, L_{K(1)}gl_1tmf(\Gamma)^{\wedge}_p)
\]
is an isomorphism onto the set of sequences $\{g_k\}_{k\ge 2}$ as in Theorem \ref{theorem:K(1)-local parameterization} satisfying the additional condition that $U_pg_k^{(p)} = g_k^{(p)}$. If $p=2$, the map
\[
\textup{Char}(Spin, L_{K(1)\vee K(2)}gl_1tmf(\Gamma)) \longrightarrow \textup{Char}(Spin, L_{K(1)}tmf(\Gamma))
\]
(cf. Definition \ref{definition:char}) is an isomorphism onto the same set of sequences. 
\end{theorem}
\begin{proof} We have a homotopy pullback square
\[
\xymatrix{
\textup{Nulls}(spin, L_{K(1)\vee K(2)}gl_1tmf(\Gamma)^{\wedge}_p)\ar[d]\ar[r] & \textup{Nulls}(spin, L_{K(2)}tmf(\Gamma)^{\wedge}_p)\ar[d]\\
\textup{Nulls}(spin, L_{K(1)}tmf(\Gamma)^{\wedge}_p) \ar[r]& \textup{Nulls}(spin, L_{K(1)}L_{K(2)}tmf(\Gamma)^{\wedge}_p)
}
\]
The upper right hand square is contractible, so the image of the map
\[
\pi_0\textup{Nulls}(spin, L_{K(1)\vee K(2)}gl_1tmf(\Gamma)^{\wedge}_p) \longrightarrow \pi_0\textup{Nulls}(spin, L_{K(1)}tmf(\Gamma)^{\wedge}_p)
\]
consists of those sequences $\{g_k\}_{k\ge 2}$ as in Theorem \ref{theorem:K(1)-local parameterization} such that $(1-U_p)g_k = 0$ (c.f. \cite[13.7]{AHR}). The fiber of this map over any point is a torsor for
\[
[\Sigma KO_p, L_{K(1)}L_{K(2)}tmf(\Gamma)^{\wedge}_p]
\]
This group vanishes for $p\ne 2$ proving the claim in this case. When $p=2$ this group is all torsion, so the discrepancy vanishes on characteristic series. 
\end{proof}
\section{Parameterizing $p$-complete Orientations}
Now we'd like to solve the extension problem
\[
\xymatrix{
g \ar[r] & gl_1S\ar[r]\ar[d] &gl_1S/g\ar@{-->}[dl]\\
&gl_1tmf(\Gamma)^{\wedge}_p&
}
\]
where $g$ is either $spin$ or $string$. We have already solved this problem for $L_{K(1)\vee K(2)}gl_1tmf(\Gamma)$, so the composite
\[
g \longrightarrow gl_1S \longrightarrow gl_1tmf(\Gamma) \longrightarrow L_{K(1)\vee K(2)}gl_1tmf(\Gamma)
\]
is nullhomotopic. Let $F$ be defined by the fiber sequence
\[
F\longrightarrow gl_1tmf(\Gamma) \longrightarrow L_{K(1)\vee K(2)}gl_1tmf(\Gamma) 
\]
then there exists a dotted arrow making the following diagram commute:
\[
\xymatrix{
g \ar[r]\ar@{-->}[d] & gl_1S\ar[r]\ar[d] &gl_1S/g\\
F\ar[r]&gl_1tmf(\Gamma)^{\wedge}_p\ar[r]&L_{K(1)\vee K(2)}gl_1tmf(\Gamma)
}
\]
and the obstruction to the existence of an orientation of $g$ is determined by the map
\[
[g,F] \longrightarrow [g, gl_1tmf(\Gamma)^{\wedge}_p]
\]
When $g = string$ as in \cite{AHR} the group $[string, F]$ vanishes so there is no obstruction. However, when $g = spin$ this group {\it never} vanishes (\ref{lemma:pi-4}) so we must understand this map. We begin in \S4.1 by collecting some necessary preliminaries on Hida's theory of ordinary modular forms. In \S4.2 we review what the work of Ando-Hopkins-Rezk can tell us about the homotopy type of $F$ in general. In \S4.3 we compute $\pi_3F$ at level 1 in terms of modular forms and show how this implies the existence of $spin$ orientations of $tmf$ away from 6. In \S4.4 we specialize to level $p$ topological modular forms and parameterize $spin$ orientations of $tmf_0(p)^{\wedge}_\ell$ when $\ell \ne 2$. Finally, in \S4.5 we treat the case when $\ell = 2$ at various prime levels. 
\subsection{Units and localization}
Given an $L_n$-local $E_\infty$ ring spectrum $R$, the spectrum of units $gl_1R$ is almost never $L_n$-local. The following theorem of Ando-Hopkins-Rezk measures the difference.
\begin{theorem}[Ando-Hopkins-Rezk, Theorem 4.11]\label{theorem:units} Let $R$ be an $L_n$-local $E_\infty$ ring spectrum and let $d_n(R)$ denote the fiber
\[
d_n(R) \longrightarrow gl_1R \longrightarrow L_ngl_1R
\]
Then $\pi_qd_n(R) = 0$ for $q\ge n+1$.
\end{theorem}

We would like to apply this to the spectrum $F$, which is the fiber of $L_{K(1)\vee K(2)}$-localization. The relationship between the two is given by the following two lemmas.
\begin{lemma} The map $L_2X \longrightarrow L_{K(1) \vee K(2)}X$ is $p$-completion.
\end{lemma}
\begin{proof} We have fiber squares
\[
\xymatrix{L_2 \ar[r]\ar[d] & L_{K(1) \vee K(2)}\ar[d]\\
L_{K(0)} \ar[r] & L_{K(0)}L_{K(1) \vee K(2)}}
 \quad \xymatrix{
L_{K(1) \vee K(2)} \ar[r]\ar[d] & L_{K(2)} \ar[d]\\
L_{K(1)} \ar[r] & L_{K(1)}L_{K(2)}
}
\]
Note that $K(n) \wedge S/p = K(n)$, and thus all $K(n)$-local spectra are $p$-complete. Since $p$-complete spectra are closed under homotopy limits, the right hand square implies that $L_{K(1)\vee K(2)}$ is $p$-complete. So it suffices to show that the top horizontal arrow of the left hand square is an equivalence, or equivalently that the bottom horizontal map is an equivalence. But this is clear because $L_{S/p}L_{K(0)} = 0$.\end{proof}
\begin{lemma} Let $R$ be a $p$-complete and $L_2$-local $E_\infty$-ring spectrum. Define $F$ as the fiber 
\[
F \longrightarrow gl_1R \longrightarrow L_{K(1)\vee K(2)}gl_1R
\]
Then $\pi_qF = 0$ for $q \ge 4$. 
\end{lemma}
\begin{proof} Recall we have defined $d_n(R)$ via the fiber sequence
\[
d_n(R) \longrightarrow gl_1R \longrightarrow L_ngl_1R
\]
and, by (\ref{theorem:units}), we know that $\pi_qd_n(R) = 0$ for $q\ge 3$. 
 
Let $0$ denote a zero object, and consider the larger diagram
\[
\xymatrix{
d_n(R) \ar[r]\ar[d] & 0\ar[d]\\
F \ar[r] \ar[d]& F'\ar[r]\ar[d] & 0\ar[d]\\
gl_1R \ar[r] & L_2 gl_1R \ar[r] & L_{K(1) \vee K(2)}gl_1R
}
\]
where all the squares are homotopy cartesian. Since $\pi_qd_n(R)= 0$ for $q\ge 3$ it suffices to show that $\pi_q F'= 0$ for $q \ge 4$. By the previous lemma, we have that the map
\[
L_2 \longrightarrow L_{K(1) \vee K(2)}
\]
is $p$-completion. Now, since 
\[
\pi_qR \cong \pi_q gl_1R \cong \pi_qL_2gl_1R
\]
for $q \ge 4$ and a spectrum is $p$-complete if and only if its homotopy groups are Ext-$p$-complete \cite[Prop. 2.5]{Bous}, we have that $\pi_qL_2gl_1R$ is Ext-$p$-complete for $q \ge 4$. 

The homotopy groups of a $p$-completion, by \cite[Prop. 2.5]{Bous}, are given by:
\[
\xymatrix{0 \ar[r] & \textup{Ext}(\mathbb{Z}/p^\infty, \pi_qL_2gl_1R) \ar[r] & \pi_qL_{K(1)\vee K(2)}gl_1R \ar[r] & \textup{Hom}(\mathbb{Z}/p^\infty, \pi_{q-1}L_2gl_1R) \ar[r] & 0}
\]
By the previous remark, since $\textup{Ext}(\mathbb{Z}/p^\infty, -)$ is Ext-$p$-completion, the first entry is just $\pi_qL_2gl_1R$ for $q \ge 4$, and the last entry vanishes for $q \ge 5$. Thus the map
\[
\pi_qL_2gl_1R \longrightarrow \pi_qL_{K(1) \vee K(2)} gl_1R
\]
is an equivalence for $q \ge 5$ and an inclusion for $q = 4$, whence $\pi_qF' = 0$ for $q \ge 4$ which completes the proof.
\end{proof}
\begin{remark} The same proof gives analogous results comparing $L_n$ and $L_{K(1)\vee \cdots \vee K(n)}$. 
\end{remark}
\subsection{Analysis of $F$}
In order to understand $F$ in the fiber sequence
\[
\xymatrix{F \ar[r] & gl_1tmf(\Gamma) \ar[r] & L_{K(1)\vee K(2)}gl_1tmf(\Gamma)}
\]
we first need to understand something about $L_{K(1)\vee K(2)}gl_1tmf(\Gamma)$. The crucial homotopy group is $\pi_4$. Recall that the Atkin operator was constructed in \S3.1, and it's effect on $q$-expansions is
\[
U_p(\sum a_nq^n) = \sum a_{np}q^n
\]
\begin{proposition} Let $p \ge 3$ be a prime not dividing the level $N$. Then there is a pullback square
\[
\xymatrix{\pi_4L_{K(1)\vee K(2)}gl_1tmf(\Gamma)^{\wedge}_p \ar[r] \ar[d]&MF_2(\Gamma)^{\wedge}_p\ar[d]\\
MF_{p, 2} \ar[r]^{1-U_p} & MF_{p, 2}}
\]
and the vertical maps are injective. 
\end{proposition}
\begin{proof} By (\ref{corollary:tmf log}) we have a homotopy fiber square
\[
\xymatrix{
L_{K(1)\vee K(2)}gl_1tmf(\Gamma)^{\wedge}_p \ar[r]\ar[d] & Tmf(\Gamma)^{\wedge}_p\ar[d]\\
L_{K(1)}tmf(\Gamma)^{\wedge}_p \ar[r]^{1-U_p}& L_{K(1)}tmf(\Gamma)^{\wedge}_p
}
\]
The result follows from the observation that $\pi_{\textup{odd}}L_{K(1)}tmf(\Gamma)^{\wedge}_p = 0$ for $p\ge 3$. To see this, recall that we have an equivalence
\[
L_{K(1)}tmf(\Gamma)^{\wedge}_p \longrightarrow tmf(\Gamma; p)^{h(\mathbb{Z}/p)^{\times}}
\]
The homotopy fixed point spectral sequence collapses to the zero line, so it suffices to make the same observation about $tmf(\Gamma;p)$. But $\mathscr{M}_{\Gamma}(p)$ is representable by a formally affine scheme when $p \ge 3$, so the descent spectral sequence for $tmf(\Gamma;p)$ collapses to the zero line and we have the result.
\end{proof}
\begin{remark} It is not true that $\pi_*tmf(\Gamma)^{\wedge}_p$ is concentrated in even degrees at large primes, in general. Indeed, when the modular curve has large genus, Serre duality forces contributions in odd, positive degrees for the homotopy of $Tmf(\Gamma)$ which then persist after taking the connective cover. Indeed, this observation is one of the reasons why the current definition of $tmf(\Gamma)$ is undesirable when the modular curve has large genus. 
\end{remark}
\begin{proposition}\label{proposition:usual inclusion} The following diagram commutes
\[
\xymatrix{
\pi_4gl_1tmf(\Gamma)^{\wedge}_p \ar[r]\ar[d]^{\cong} & \pi_4L_{K(1)\vee K(2)}gl_1tmf(\Gamma)^{\wedge}_p\ar[d]\\
MF_2(\Gamma) \otimes \mathbb{Z}_p \ar[r]_{1 - \psi^p} & MF_{p,2}(\Gamma)
}
\]
\end{proposition}
Putting these two together we can often compute $\pi_3F$ in practice. For now, we only remark that it is torsion-free.
\begin{corollary}\label{corollary:torsion-free} Let $p\ge 5$, then the group $\pi_3F$ is torsion free.
\end{corollary}
\begin{proof} By the above proposition and a short diagram chase, $\pi_3F$ injects into the cokernel of the map
\[
1-\psi^p: MF_2(\Gamma) \otimes \mathbb{Z}_p \longrightarrow MF_{p,2}(\Gamma)
\]
This is torsion-free since the reduction modulo $p$ of this map is injective and $MF_{p,2}(\Gamma)$ is torsion-free. 
\end{proof}
\subsection{The Discrepancy Map and Weight 2 Modular Forms}
We have a homotopy pullback diagram of spectra under $gl_1S$
\[
\xymatrix{
gl_1tmf(\Gamma)^{\wedge}_p \ar[r]\ar[d] &L_{K(1)\vee K(2)}gl_1tmf(\Gamma)\ar[d] \\
0 \ar[r] & \Sigma F
}
\]
where the structure map $gl_1S\longrightarrow \Sigma F$ is defined as the composite $gl_1S \rightarrow L_{K(1)\vee K(2)}gl_1tmf(\Gamma)\rightarrow \Sigma F$. Thus, we get a homotopy pullback diagram
\[
\xymatrix{
\textup{Nulls}(spin, gl_1tmf(\Gamma)^{\wedge}_p) \ar[r]\ar[d] &\textup{Nulls}(spin, L_{K(1)\vee K(2)}gl_1tmf(\Gamma))\ar[d]^{\circleit{B}} \\
\ast \ar[r]^{\circleit{A}} & \textup{Nulls}(spin, \Sigma F)
}
\]
 and we need to understand the maps \circleit{A} and \circleit{B}. By definition, the first map hits the orientation corresponding to the composite
 \[
 gl_1S/spin \longrightarrow 0 \longrightarrow \Sigma F
 \]
 using this to trivialize the torsor we get the following tautological result.
 \begin{proposition} Using the zero map $gl_1S/spin \longrightarrow \Sigma F$ to trivialize the torsor $\textup{Nulls}(spin, \Sigma F)$, the image of the map $\circleit{A}$ is the zero map
 \[
 bspin \longrightarrow \Sigma F
 \]
 \end{proposition}
To detect elements in $\textup{Nulls}(spin, \Sigma F)$ we need only check one homotopy group.
\begin{lemma}\label{lemma:pi-4} Evaluation at $\pi_4$ gives an isomorphism
\[
[bspin, \Sigma F ] \cong \pi_4\Sigma F
\]
\end{lemma}
\begin{proof} Since $bspin$ has no homotopy below $\pi_4$, we have $[bspin, \Sigma F] = [bspin, \tau_{\ge 4}\Sigma F]$. Since $\Sigma F$ has no homotopy above $\pi_4$ this is $[bspin, H\pi_4\Sigma F]$ and the result follows from Hurewicz. 
\end{proof}
Let's recall how to associate a sequence of modular forms to a nullhomotopy as in (\ref{theorem:unscrewable orientations}). Given an element in $\pi_0\textup{Nulls}(spin, L_{K(1)\vee K(2)}gl_1tmf(\Gamma))$ we have a map
\[
gl_1S/spin \longrightarrow L_{K(1)\vee K(2)}gl_1tmf(\Gamma)^{\wedge}_p
\]
Upon $K(1)$-localization, we get a composite
\[
\xymatrix{
KO_p \ar[r]_-{b_c}^-{\cong} & L_{K(1)}gl_1S/spin \ar[r] & L_{K(1)}gl_1tmf(\Gamma)^{\wedge}_p \otimes \mathbb{Q}
}
\]
and the image of generators of $\pi_{4k}KO_p$, suitably divided by 2 as in (\ref{definition:char}), give us a sequence of modular forms of the shape $(1-c^{2k})(1-\psi^p)g_{2k}$. The sequence $\{g_{2k}\}$ is the associated sequence.   
\begin{lemma} Let $c \in \mathbb{Z}_{\ge 0}$ be a topological generator of $\mathbb{Z}_p$. Let $b_c: KO_p \cong L_{K(1)}gl_1S/spin$ be the equivalence constructed in (\ref{theorem:diagram}). Then there exists an element $x \in \pi_4 gl_1S/spin$ whose image under $K(1)$-localization is $b_c(x_4)$, where we recall that $x_4 \in \pi_4KO_p$ is a generator.
\end{lemma}
\begin{proof} Since $\pi_4S = 0$, we have an exact sequence
\[
\xymatrix{ 0\ar[r] & \pi_4gl_1S/spin \ar[r] & \pi_4bspin \ar[r] & \pi_3gl_1S}
\]
Let $x_4$ denote the generator of $\pi_4bspin$. Then, since $\pi_3S = \mathbb{Z}/24$, $24x_4$ is hit by an element $x \in \pi_4gl_1S/spin$. This has the desired properties by inspection of the diagram (\ref{theorem:diagram}). 
\end{proof}
\begin{theorem}\label{theorem:p-complete} Suppose that $\pi_3F$ is torsion-free (e.g. if $p\ge 5$ by (\ref{corollary:torsion-free})). Identifying elements in $$\pi_0\textup{Nulls}(spin, L_{K(1)\vee K(2)}gl_1tmf(\Gamma))$$ with certain sequences of $p$-adic modular forms as reviewed above, the map \circleit{B} takes an element $\{g_k\}_{k\ge 2}$ to the image of \circleit{A} if and only if $g_2$ is in the image of the usual inclusion
\[
MF_2(\Gamma) \otimes \mathbb{Z}_p \longrightarrow MF_{p, 2}(\Gamma)
\]
\end{theorem}
\begin{proof} Consider the following diagram:
\[
\xymatrix{
gl_1S/spin \ar[r]\ar[d] & bspin \ar[d]\\
L_{K(1)\vee K(2)}gl_1tmf(\Gamma)^{\wedge}_p \ar[r] & \Sigma F
}
\]
By the previous lemma, there is an element $x \in \pi_4gl_1S/spin$ whose image under $K(1)$-localization is $b_c(x_4)$.  Let $\phi$ denote the image of $x_4\in \pi_4bspin$ under the right hand vertical map, and let $y$ denote the image of $x$ under the left hand vertical map. Then commutativity of the diagram, together with the diagram of (\ref{theorem:diagram}), gives the picture
\[
\xymatrix{
x\ar[r]\ar[d] & (1-c^2)x_4\ar[d]\\
y \ar[r]& (1-c^2)\phi
}
\]
When $p\ge 5$, $1-c^2$ is a $p$-adic unit. Indeed, if $1-c^2 \equiv 0$ mod $p$, then $c$ would be $\pm 1$ mod $p$ and so wouldn't generate $\mu_{p-1}/ \{\pm 1\}$ when $p\ge 5$.  Thus, when $p\ge 5$ or $\pi_3F$ is torsion-free, $\phi = 0$ if and only if $(1-c^2)\phi = 0$, and this happens if and only if $y$ maps to zero in $\pi_3F$. Equivalently, we're asking that $y$ be in the image of 
\[
\pi_4gl_1tmf(\Gamma)^{\wedge}_p \longrightarrow \pi_4L_{K(1)\vee K(2)}gl_1tmf(\Gamma)^{\wedge}_p
\]
which we may check after composition with the injective map
\[
\pi_4L_{K(1)\vee K(2)} gl_1tmf(\Gamma)^{\wedge}_p \longrightarrow \pi_4L_{K(1)}tmf(\Gamma)^{\wedge}_p \otimes \mathbb{Q}. 
\]
By definition of $x$, the element $y$ maps to the element $2(1-c^2)(1-\psi^p)g_2$, and by (\ref{proposition:usual inclusion}) this lies in the image of the map from $\pi_4gl_1tmf(\Gamma)^{\wedge}_p$ if and only if $g_2$ came from an element in $MF_2 \otimes \mathbb{Z}_p$. 
\end{proof}
\subsection{The case $p=3$}
When $p=3$ trouble may occur if $\pi_3F$ has $3$-torsion. This doesn't happen as soon as the moduli problem is representable.
\begin{lemma} If $\mathscr{M}_{\Gamma}[1/N]$ is representable by a scheme, $\pi_3F$ is torsion-free. In particular, this is true for $\Gamma = \Gamma_0(2)$. 
\end{lemma}

Recall we have an exact sequence
\[
\pi_4gl_1tmf(\Gamma) \longrightarrow \pi_4 L_{K(1)\vee K(2)}gl_1tmf(\Gamma)^{\wedge}_3 \longrightarrow \pi_3F \longrightarrow \pi_3gl_1tmf(\Gamma)^{\wedge}_3
\]
Thus, it suffices to show that the first map has torsion-free cokernel and the last group is torsion-free. The former claim has the same proof as (\ref{corollary:torsion-free}). The latter claim is the context of the next lemma, which completes the proof of the theorem. 
\begin{lemma} When $\left(\mathscr{M}_\Gamma\right)^{\wedge}_3$ is representable by a scheme, the group $\pi_3tmf(\Gamma)^{\wedge}_3$ is torsion-free. 
\end{lemma}
\begin{proof} The only possible torsion is in $\pi_1$, see \cite[6.4]{HL}. 
\end{proof}

\subsection{The case $p=2$}
At the prime $2$, torsion appears in even simple examples:
\begin{lemma} The torsion in $\pi_3F$ is at most a $\mathbb{Z}/2$. 
\end{lemma}
\begin{proof} We have an exact sequence
\[
\pi_4gl_1tmf_0(3) \longrightarrow \pi_4L_{K(1)\vee K(2)}gl_1tmf_0(3) \longrightarrow \pi_3F \longrightarrow \pi_3gl_1tmf_0(3)
\]
and the last term is $\mathbb{Z}/2$ by \cite{MR}. The cokernel of the first map is also torsion-free by the argument as in the previous section, thus $\pi_3F$ can, at most, contain a $\mathbb{Z}/2$ as torsion.
\end{proof}
\begin{corollary} Let $\textup{bswing}$ denote the fiber of the map
\[
bspin \stackrel{w_4}{\longrightarrow} \Sigma^4H\mathbb{Z}/2
\]
Then, for every prime $p \ne 3$, $\textup{MSwing}$ admits a $tmf_0(3)^{\wedge}_p$-orientation for every sequence $\{g_k\}_{k \ge 2} \in MF_{p,k}(\Gamma_0(3))\otimes \mathbb{Q}$ of modular forms satisfying the conditions of (\ref{theorem:unscrewable orientations}) and such that $g_2 \in MF_2(\Gamma_0(3))$. 
\end{corollary}
\begin{proof} Combine the previous lemma with Theorem \ref{theorem:p-complete}. 
\end{proof}

We suspect a similar statement can be made for $tmf_0(5)$, though perhaps a larger power of 2 is necessary. 
\section{Building Orientations}
We can now combine our results thus far to prove the main theorem.
\begin{proof}[Proof of Theorem \ref{theorem:main}] Let $M$ be any of the Thom spectra as in the statement of the theorem. We have a homotopy pullback square
\[
\xymatrix{
\textup{Map}_{E_\infty}(M, tmf(\Gamma)) \ar[r]\ar[d] & \prod_p\textup{Map}_{E_\infty}(M, tmf(\Gamma)^{\wedge}_p)\ar[d]\\
\textup{Map}_{E_\infty}(M, tmf(\Gamma)_{\mathbb{Q}}) \ar[r]& \textup{Map}_{E_\infty}\left(M, \left(\prod_ptmf(\Gamma)^{\wedge}_p\right)_{\mathbb{Q}}\right)
}
\]
and $\pi_1$ of the lower right hand side is zero, so we get a pullback square on connected components. This mostly completes the proof, except that we are using the Eichler-Schimura relation and the fact that $U_p\psi^p = p^k$ to see that:
\[
(1-U_p)(1-\frac{1}{p}\psi^p) = 0 \iff 1+p^{k-1} - T_p = 0
\]
We are free to replace the condition (\ref{theorem:K(1)-local parameterization}.3) on the total mass of the measures by the congruence condition (\ref{theorem:main}(d)) once we have shown (\S 5.2) that there is at least one orientation whose characteristic series satisfies this congruence (c.f. \cite[10.7]{AHR}). 
\end{proof}
We will use this theorem below to construct examples of genera valued in topological modular forms with level structure. Before we do, we must recall some preliminary algebraic results concerning Eisenstein series. 
\subsection{Eisenstein Series and the Eisenstein Measure}
Let $\mathscr{M}^{\infty}_{\Gamma} \subset \mathscr{M}_{\Gamma}$ denote the degenerate locus. This is a relative, effective Cartier divisor in $\mathscr{M}_{\Gamma}$ over $\mathbb{Z}$, and so corresponds to an exact sequence
\[
\xymatrix{
0 \ar[r]&\mathcal{O}_{\mathscr{M}_{\Gamma}}(-\textup{cusps}) \ar[r] & \mathcal{O}_{\mathscr{M}_\Gamma} \ar[r] & \mathcal{O}_{\mathscr{M}^{\infty}_{\Gamma}} \ar[r] & 0
}
\]
which gives an exact sequence
\[
\xymatrix{
0 \ar[r]&\omega^{\otimes k}_{\mathscr{M}_{\Gamma}}(-\textup{cusps}) \ar[r] & \omega^{\otimes k}_{\mathscr{M}_\Gamma} \ar[r] & \omega^{\otimes k}_{\mathscr{M}^{\infty}_{\Gamma}} \ar[r] & 0
}
\]
The Hecke correspondence preserves the divisor $\mathscr{M}^{\infty}_{\Gamma}$ so this sequence is Hecke equivariant.
\begin{theorem}\label{theorem:hecke-splitting} Let $I$ denote the image of the map $H^0(\mathscr{M}_{\Gamma}, \omega^{\otimes k}) \longrightarrow H^0(\mathscr{M}^{\infty}_{\Gamma}, \omega^{\otimes k})$. When $k\ge 2$, there is a unique, Hecke-equivariant splitting of the exact sequence
\[
\xymatrix{
0 \ar[r]&H^0(\mathscr{M}_{\Gamma}, \omega^{\otimes k}_{\mathscr{M}_{\Gamma}}(-\textup{cusps})) \ar[r] & H^0(\mathscr{M}_{\Gamma},\omega^{\otimes k}) \ar[r] & I \ar[r] & 0
}
\]
\end{theorem}
\begin{proof} This is classical, see \cite{Eme} for a discussion and further references.   
\end{proof}
The image of the splitting guaranteed in (\ref{theorem:hecke-splitting}) is called the Eisenstein subspace of $M_k(\Gamma)$ and denotes $\textup{Eisen}_k(\Gamma)$. 

We have the following dimension formulae for distinct primes $p$ and $\ell$. 
\[
\textup{dim}(\textup{Eisen}_k(\Gamma_0(p))) = \begin{cases} 2 & k>2\\
1 &k=2
\end{cases}
\]
\[
\textup{dim}(\textup{Eisen}_k(\Gamma_0(p\ell)) = \begin{cases} 4 &k>2\\
3& k=2
\end{cases}
\]
An explicit basis is given as follows. The $q$-expansion of the unnormalized Eisenstein series of weight $k$ and level 1, for $k\ge 4$ even is given by
\[
G_k = \frac{B_k}{2k} + \sum_{n\ge 1}\left(\sum_{d \vert n} d^{k-1}\right) q^n
\]
When $k=2$, the formal power series still makes sense, but it is not a modular form. However, for any prime $p$,
\[
G_2(q) - pG_2(q^p)
\]
is the $q$-expansion of a modular form of level $\Gamma_0(p)$. 
\begin{proposition} For $k>2$ even, the modular forms $G_k$ and $\frac{1}{p^k}G_k\vert_{\psi^p}$ form a basis for $\textup{Eisen}_k(\Gamma_0(p))$, and the modular forms $G_k, G_k^{(p)}, G_k^{(\ell)}, G_k^{(p)(\ell)}$ form a basis for $\textup{Eisen}_k(\Gamma_0(p\ell))$. When $k=2$, we must remove $G_2$ from the list. 
\end{proposition}
\begin{proof} See, for example, \cite[4.5.2]{DS}. 
\end{proof}
We now compute the action of $U_p$ on the space of Eisenstein series. 

\begin{proposition} Let $k>2$. Using the basis $\{G_k^{(p)}, G_k - \frac{1}{p^k}G_k\vert_{\psi^p}\}$ for $\textup{Eisen}_k(\Gamma_0(p))$ the action of $U_p$ is via the matrix
\[
\begin{pmatrix}
1 & 0\\
0 & p^{k-1}
\end{pmatrix}
\]
The action of $U_p$ on $\textup{Eisen}_k(\Gamma_0(p\ell))$ preserves the two dimensional summands $\{G_k^{(p)}, G_k - \frac{1}{p}G_k\vert_{\psi^p})\}$ and $\{G_k^{(p)(\ell)}, G_k^{(\ell)} - \frac{1}{p^k}G_k^{(\ell)}\vert_{\psi^p}\}$ and acts on each via the same matrix:
\[
\begin{pmatrix}
1 & 0 \\
0 & p^{k-1}
\end{pmatrix}
\]
In particular, the fixed points of $U_p$ acting on $\textup{Eisen}_k(\Gamma_0(p\ell))$ are spanned by $\{G_k^{(p)}, G_k^{(p)(\ell)}\}$. 
\end{proposition}
\begin{proof} This is an exercise in the Eichler-Shimura relation and the relation $U_p\psi^p = p^k$. 
\end{proof}
Finally, we recall the existence of the Eisenstein measure. 
\begin{theorem}[Katz] Let $p$ and $\ell$ be distinct primes and fix $c \in \mathbb{Z}_p^{\times}/\{\pm 1\}$. Then there is a measure $\mu_c$ on $\mathbb{Z}_p^{\times}/\{\pm 1\}$ valued in $V_{\infty}(\Gamma)$ such that
\[
\int_{\mathbb{Z}_p^{\times}/\{\pm 1\}} x^{2k} d\mu_c = (1-c^{2k})(1-\frac{1}{p}\psi^p)G^{(\ell)}_{2k}
\]
\end{theorem}
\begin{proof} Take the measure constructed in \cite[3.3.3, 3.4.1]{K2} with $b=1$ and $a=c$. The extra factor of 2 needed to get a measure on $\mathbb{Z}_p^{\times}/\{\pm 1\}$ follows as in \cite[10.10]{AHR}. 
\end{proof}
\subsection{Examples}
\begin{theorem} There exist, up to homotopy, unique $E_\infty$-ring maps
\[
\sigma_{\textup{Och}}, \sigma_{\textup{WSig}}: \textup{MSpin} \longrightarrow tmf_0(2)
\]
refining the Ochanine genus and Witten signature, respectively, as defined in (\ref{definition:genera examples}). In either case, evaluating at the two different cusps of $\mathscr{M}_0(2)$ gives two different genera
\[
\textup{MSpin} \longrightarrow KO[1/2]
\]
one of which is the $\widehat{A}$-genus and the other evaluates to $\textup{Sign}(M)/2^{2d}$ on an oriented manifold of dimension $4d$. 
\end{theorem}
\begin{proof} Uniqueness follows from Theorem \ref{theorem:main} so we need only show existence. This, in turn, follows from the existence of the Eisenstein measure, and the results in the appendix (\ref{proposition:congruences}). 
\end{proof}
\begin{remark} Consider the following diagram:
\[
\xymatrix{
\textup{MString}\ar[d]\ar[r]^{\sigma} & tmf\ar@{-->}[d]\\
\textup{MSpin} \ar[r]_{\textup{Ochanine}} & tmf_0(2)
}
\]
The only map we know about on the right hand side comes from the forgetful map on the corresponding moduli. This map does \emph{not} yield a commutative diagram, even rationally. Indeed, this follows immediately from the characteristic series calculations in (\ref{theorem:characteristic-series-formulae}). In order to get a commutative diagram, we would need to understand the functoriality of $\mathcal{O}^{\textup{top}}$ with respect to isogenies of formal groups. As far as the author knows, this is not treated in the literature. 
\end{remark}

The previous result makes one wonder if the signature itself is an $E_\infty$-ring map, a result that does not seem to appear in the literature, but is nevertheless an easy consequence of the work in \cite{AHR}. We record the result here. 
\begin{theorem} There exists, up to homotopy, a unique $E_\infty$-ring map
\[
\textup{MSpin} \longrightarrow KO
\]
refining the $L$-genus. 
 \end{theorem}
\begin{proof} This follows already from the work in \cite{AHR} once one checks the hypotheses on the characteristic series, as in the appendix (\ref{proposition:congruences}). 
\end{proof}
\begin{theorem}\label{theorem:tmf-away-from-6} There exist $E_\infty$-ring maps
\[
\textup{MSpin} \longrightarrow tmf[1/6]
\]
but none of these make the following diagram commute
\[
\xymatrix{
\textup{MString}\ar[d]\ar[r]^{\sigma} & tmf[1/6]\\
\textup{MSpin} \ar@{-->}[ur]_{\not\exists} &}
\]
\end{theorem}
\begin{proof} First we show that the given diagram cannot commute. Suppose it did, then we would have a commutative diagram
\[
\xymatrix{
\textup{MString}\ar[d]\ar[r]^{\sigma} & tmf[1/6]\ar[d]\\
\textup{MSpin} \ar@{-->}[ur]\ar[r]_{\widehat{A}} &KO^{\wedge}_5}
\]
Indeed, the composite
\[
\textup{MString} \longrightarrow tmf[1/6] \longrightarrow KO^{\wedge}_5
\]
has the same characteristic series as the restriction of the $\widehat{A}$-genus to String-manifolds, so it must be the $\widehat{A}$-genus by \cite[7.12]{AHR}. Moreover, the $K(1)$-localization of $\textup{MString}$ agrees with the $K(1)$-localization of $\textup{MSpin}$, c.f. \cite[2.3.1]{Hov}, so this determines the bottom map making the diagram commute. On the other hand, since $\pi_4tmf[1/6] = 0$, this would imply that all $4$-dimensional Spin-manifolds have trivial $\widehat{A}$-genus (after completing at $5$). This is false, for example if $K$ is a K3 surface then $\widehat{A}(K) =2$.

Now we prove existence of such a genus (there are many). It is enough to write down a suitable sequence of modular forms, $\{g_k\}_{k \ge 2}$. We will define
\[
g_k:= G_k - q_kG_k
\]
with $q_k \in \mathbb{Z}$. By Theorem \ref{theorem:main} and \cite[Def. 27, Thm 40]{Spr}, we are reduced to constructing a sequence $\{q_k\}_{k \ge 2}$ of integers such that: (i) $q_{2k+1} = 0$ for $k\ge 1$, (ii)  for all primes $p$ there is a measure $\mu_p \in \textup{Meas}(\mathbb{Z}_p^{\times}/\{\pm 1\}, \mathbb{Z}_p)$ such that 
\[
q_{2k} = \int_{\mathbb{Z}_p^{\times}/\{\pm 1\}} x^{2k} d\mu_p, \quad k\ge 1
\]
and
\[
0 = \int_{\mathbb{Z}_p^{\times}/\{\pm 1\}} x^{2k} d\mu_p
\]
and (iii) $q_2 = 1$. By \cite[Thm 31]{Spr}, there are uncountably many such sequences satisfying (i) and (ii) and we are free to specify the value of $q_2$ however we like. 
\end{proof}
The next natural set of examples are certain complex orientations $\textup{MU} \longrightarrow tmf_1(N)$ defined by Hirzebruch, which specialize to a version of the $\chi_y$ genus at the ramified cusp. In these cases it is more difficult to check the corresponding integrality conditions on the characteristic series, so we leave this to a future paper.   
\appendix
\section{Formulae for characteristic series}
It is not usually the case that the Hirzebruch characteristic series of a genus is given in the form $\textup{exp}\left(2\sum t_k x^k/k!\right)$. For the convenience of the reader, we manipulate the formulae for the genera appearing in the paper into this form. We begin by fixing some notation and recalling the definitions in the literature.

\begin{definition}\label{definition:genera examples} The following define genera for $MSO_*$.
\begin{enumerate}
\item $L$-genus
\[
\log_{\textup{Sign}}(x) = \sum_{n \ge 1} \frac{x^{2n+1}}{2n+1}
\]
\item $\widehat{A}$-genus
\[
\exp_{\widehat{A}}(u) = 2\sinh(u/2)
\]
\item Ochanine genus
\[
\log_{\textup{Och}}(x) = \int_0^x \frac{dt}{\sqrt{1 - 2\delta t^2 + \epsilon t^4}}
\]
\item Witten genus
\[
\frac{u}{\exp_{\textup{Wit}}(u)} = \frac{u/2}{\sinh(u/2)} \prod_{n=1}^\infty \frac{(1-q^n)^2}{(1-q^ne^u)(1-q^ne^{-u})}
\]
\item Witten signature
\[
\frac{u}{\exp_{\textup{WSig}}(u)}= \frac{u/2}{\tanh(u/2)} \prod_{n=1}^{\infty} \left(\frac{1+q^ne^u}{1-q^ne^u} \cdot \frac{1+q^ne^{-u}}{1-q^ne^{-u}}\right)\bigg/ \left(\frac{1+q^n}{1-q^n}\right)^2
\]

\end{enumerate}
\end{definition}

\begin{theorem}\label{theorem:characteristic-series-formulae} We have the following identities of formal power series (where we have some redundant factors of 2 we've added to put them in a more useable form for our purposes.) 
\begin{eqnarray}
\frac{u}{\exp_{\textup{Sign}}(u)} = \exp\left( 2 \sum_{k\ge 2} \frac{2^{k+1}(2^{k-1} - 1)}{2k} B_k \frac{u^k}{k!}\right)\\
\frac{u}{\exp_{\widehat{A}}(u)} = \exp\left(2\sum_{k \ge 2} \frac{-B_k}{2k} \frac{u^k}{k!}\right)\\
\frac{u}{\exp_{\textup{WSig}}(u)} = \exp\left( 2 \sum_{k \ge 2} 2G_k^{(2)} \frac{u^k}{k!}\right)\\
\frac{u}{\exp_{\textup{Wit}}(u)} = \exp\left(2\sum_{k \ge 2} G_k \frac{u^k}{k!}\right)\\
\frac{u}{\exp_{\textup{Och}}(u)} = \exp\left(2\sum_{k \ge 2} \widetilde{G}_k \frac{u^k}{k!}\right)
\end{eqnarray}
\end{theorem}
\begin{proof} For (2) and (4) see \cite[10.2, 10.9]{AHR}. For (5) see \cite{Zag}. Zagier also states the result for the Witten's signature, but for completeness we include the derivation of this here. The formula (1) for the characteristic series of the signature follows from evaluation at $q=0$ and substitution of $u$ for $u/2$, once we note that
\[
\log_{\textup{Sign}}(x) = \frac{1}{2}\left(\log(1+x) - \log(1-x)\right)
\]
and hence
\[
\exp_{\textup{Sign}}(u) = \frac{e^u - e^{-u}}{e^u + e^{-u}} = \tanh(u)
\]
So we are left with formula (3). By the definition, 
\[
\log\left(\frac{u}{\exp_{\textup{WSig}}(u)}\right) = \log\left(\frac{u/2}{\tanh(u/2)}\right) + \sum_{n \ge 1}\left(\log\left(\frac{1+q^ne^u}{1-q^ne^u}\right) + \log\left(\frac{1+q^ne^{-u}}{1-q^ne^{-u}}\right) - 2\log\left(\frac{1+q^n}{1-q^n}\right)\right)
\]
Using the Taylor series for $\log(1-x)$, we can rewrite this as
\begin{eqnarray*}
\log\left(\frac{u}{\exp_{\textup{WSig}}(u)}\right) &=&  \log\left(\frac{u/2}{\tanh(u/2)}\right) + 2\sum_{n \ge 1} \sum_{d \textup{ odd}} (e^{ud} + e^{-ud} -2) \frac{q^{nd}}{d}\\
&=& \log\left(\frac{u/2}{\tanh(u/2)}\right) + 2\sum_{n\ge 1}\sum_{d \textup{ odd}}\frac{q^{nd}}{d}\left(-2+ \frac{1}{k!}\sum_{k \ge 0} d^k(u^k + (-1)^ku^k)\right) \\
&=&\log\left(\frac{u/2}{\tanh(u/2)}\right) + 2\sum_{n\ge 1}\sum_{d \textup{ odd}}\frac{q^{nd}}{d}\left(-2+ \frac{2}{k!}\sum_{k \textup{ even}} d^ku^k\right)
\end{eqnarray*}
When $k=0$, we are left only with the constant ($q=0$) term from cancellation. When $k>0$ and even the coefficient of $u^k/k!$ is
\[
2\sum_{d\textup{ odd}}d^{k-1}\sum_{n\ge 1} q^{nd}
\]
which is the same as
\[
2\sum_{n \ge 1}q^n \sum_{d\vert n, d \textup{ odd}} d^{k-1} = G_k^{(2)}(q) - G_k^{(2)}(0)
\]
Thus it suffices to show that
\[
\log\left(\frac{u/2}{\tanh(u/2)}\right) = 2\sum_{k \ge 2}\frac{2^{k} - 2}{2k}B_k \frac{u^k}{k!}
\]
Recall that $\tanh'(x) = \frac{1}{\cosh^2(u)}$, whence
\begin{eqnarray*}
d\log\left(\frac{u/2}{\tanh(u/2)}\right) &=& \frac{1}{u} - \frac{1}{2\sinh(u)\cosh(u)}\\
&=& \frac{1}{u} - \frac{2}{e^u - e^{-u}}\\
&=& \frac{1}{u} - 2\frac{e^u}{e^{2u} - 1}\\
&=& \frac{1}{u} - 2\frac{e^u +1}{e^{2u} -1} + \frac{2}{e^{2u} -1}\\
&=& \frac{1}{u} -  \frac{2}{e^u -1} + \frac{2}{e^{2u} - 1}
\end{eqnarray*}
Now recall the exponential generating function for the Bernoulli numbers is $u/(e^u - 1)$, so we get
\begin{eqnarray*}
d\log\left(\frac{u/2}{\tanh(u/2)}\right)&=& \frac{1}{u} - 2\sum_{k\ge 0}\frac{B_k}{k} \frac{u^{k-1}}{(k-1)!} + 2\sum_{k \ge 0} 2^{k-1}\frac{B_k}{k}\frac{u^{k-1}}{(k-1)!}\\
&=& 2\sum_{k \ge 2} \frac{2^{k-1} - 1}{k} B_k \frac{u^{k-1}}{(k-1)!}
\end{eqnarray*}
(Here we've noted that the coefficient of $1/u$ and the constant term are both zero from cancellation.) Integrating and multiplying by $\frac{2}{2}$ gives the result as desired. 
\end{proof}
\begin{proposition}\label{proposition:congruences} We have the following congruences
\begin{eqnarray}
-\frac{B_k}{2k} \equiv \frac{2(2^{k-1}-1) B_k}{2k} \quad \textup{mod }\mathbb{Z}[1/2], \quad k\ge 2\\
-\frac{B_k}{2k} \equiv \frac{2^{k+1}(2^{k-1}-1) B_k}{2k} \quad \textup{mod }\mathbb{Z}, \quad k\ge 2\\
G_k \equiv \widetilde{G}_k\equiv 2G_k^{(2)} \quad\textup{mod }\mathbb{Z}[1/2], \quad k \ge 2
\end{eqnarray}
\end{proposition}
\begin{proof} Let $m(k)$ denote the denominator of $\frac{B_k}{2k}$ for $k$ even. Then, by \cite[2.7]{Ad}, for $p$ odd we have
\[
\nu_p(m(k)) = \begin{cases}
0 & k \not\equiv 0 \textup{ mod }(p-1)\\
\nu_p(k) +1 & k \equiv 0 \textup{ mod }(p-1)
\end{cases}
\]
So for $p$ odd with $(p-1) \vert k$, write $k = (p-1)p^{\nu_p(k)}u$. Then, since $\phi(p^{\nu_p(k) + 1}) = (p-1)p^{\nu_p(k)}$, we have
\[
2^k = (2^{(p-1)p^{\nu_p(k)}})^u \equiv 1 \textup{ mod }p^{\nu_p(k)+1}
\]
thus
\[
(2^k - 2) + 1 \equiv -1 + 1 \equiv 0 \textup{ mod }p^{\nu_p(k) +1}
\]
and we see that $$\frac{2(2^{k-1} -1) +1}{2k}B_k$$ has only powers of $2$ in the denominator. The power of 2 that appears is, again by \cite[2.7]{Ad}, $2^{m+2}$ where $k = 2^mu$. This divides $2^{k+1}$ only when $k+1\ge m+2$, which happens as soon as $k>0$. The final congruences follow from the first one by comparing constant terms on $q$-expansions (the other coefficients are all integers).  
\end{proof}

\bibliographystyle{alpha}
\bibliography{ochanine}
\end{document}